\documentclass[reqno]{amsart}

\usepackage{amssymb,amsthm,amsfonts,amstext,amsmath}
\usepackage{mathtools}
\usepackage{fontawesome5}
\usepackage{eso-pic}
\usepackage{charter}
\usepackage{listings}
\usepackage{color}
\usepackage{bbm}

\definecolor{dkgreen}{rgb}{0,0.6,0}
\definecolor{gray}{rgb}{0.5,0.5,0.5}
\definecolor{mauve}{rgb}{0.58,0,0.82}
\lstdefinelanguage{none}{
	identifierstyle=
}

\usepackage{comment}
\usepackage{graphicx}
\usepackage{enumerate}
\numberwithin{equation}{section}
\usepackage{color}
\usepackage{float}
\usepackage[colorlinks=true,linkcolor=blue,citecolor=magenta]{hyperref}

\makeatletter
\def\namedlabel#1#2{\begingroup
	#2%
	\def\@currentlabel{#2}%
	\phantomsection\label{#1}\endgroup
}
\makeatother

\usepackage{mathrsfs}
\usepackage{dsfont}
\usepackage{tikz}
\usetikzlibrary{decorations.pathreplacing, shapes.geometric, positioning}
\usepackage{tikz-3dplot}

\usepackage[utf8]{inputenc}
\usepackage[T1]{fontenc}
\usetikzlibrary{backgrounds}
\usetikzlibrary{patterns,fadings}
\usetikzlibrary{arrows,decorations.pathmorphing}
\usetikzlibrary{calc}
\definecolor{light-gray}{gray}{0.95}
\usepackage{float}
\usepackage[colorlinks=true,linkcolor=blue,citecolor=magenta]{hyperref}
\usepackage[a4paper, margin=2.cm]{geometry}
\def\centerarc[#1](#2)(#3:#4:#5){\draw[#1] ($(#2)+({#5*cos(#3)},{#5*sin(#3)})$) arc (#3:#4:#5);}
\usepackage{placeins}
\usepackage{ulem}\usepackage{todonotes}

\newcommand{\msf}[1]{{\mathsf #1}}
\newcommand{\mc}[1]{{\mathcal #1}}
\newcommand{\mf}[1]{{\mathfrak #1}}

\newcommand{\bb}[1]{{\mathbb #1}}

\newcommand{\eps}{\varepsilon}

\newcommand{\one}{\mathbbm{1}}

\newcommand{\Prob}[2]{\mathbb{P}_{#1}\left(#2\right)}
\newcommand{\Expec}[2]{\mathbb{E}_{#2}\left[#1\right]}
\newcommand{\CExpec}[3]{\mathbb{E}_{#2}\left[#1 \middle| #3\right]}
\newcommand{\gn}[3]{{#1}_{#3}\left(#2\right)}

\newcommand{\pfrac}[2]{\genfrac{}{}{}{1}{#1}{#2}}

\newcommand{\dd}{{\bf d}}

\newcommand{\f}[1]{f\left(#1\right)}
\newcommand{\g}[1]{g\left(#1\right)}
\newcommand{\of}[1]{\text{o}\left(#1\right)}

\let\oldtocsection=\tocsection
\let\oldtocsubsection=\tocsubsection
\let\oldtocsubsubsection=\tocsubsubsection
\renewcommand{\tocsection}[2]{\hspace{0em}\oldtocsection{#1}{#2}}
\renewcommand{\tocsubsection}[2]{\hspace{1em}\oldtocsubsection{#1}{#2}}
\renewcommand{\tocsubsubsection}[2]{\hspace{2em}\oldtocsubsubsection{#1}{#2}}
\DeclareRobustCommand{\SkipTocEntry}[5]{}

\newtheorem{theorem}{Theorem}[section]
\newtheorem{lemma}[theorem]{Lemma}

\newtheorem{corollary}[theorem]{Corollary}

\newtheorem{definition}[theorem]{Definition}

\begin{document}

\title[Scaling limit of boundary random walks]{Scaling limit of boundary random walks: A martingale problem approach}

\author[Arroyave, J.C.]{Juan Carlos Arroyave}
\address{IMPA\\
	Estrada Dona Castorina, 110\\
	Jardim Botânico\\
	CEP 22460-320\\
	Rio de Janeiro, Brazil}
\email{juan.arroyaveb@gmail.com}

\author[Barros, E.]{Eldon Barros}
\address{IMPA\\
	Estrada Dona Castorina, 110\\
	Jardim Botânico\\
	CEP 22460-320\\
	Rio de Janeiro, Brazil}
\email{eldon.barros@impa.br}

\author[Pimenta, E.]{Eduardo Pimenta}
\address{UFBA\\
	Instituto de Matem\'atica e estat\'istica\\ 
	Campus de Ondina, Av. Milton Santos, S/N.\\ 
	CEP 40170-110\\
	Salvador, Brazil}
\email{sampaieduardo@gmail.com}

\thanks{E.B. would like to thank CAPES for the support through a Master's degree scholarship. A.J.C. thanks CAPES for the support through a Ph.D. scholarship.}

\begin{abstract}
We establish the scaling limit of a class of boundary random walks to the full spectrum of Brownian-type processes on the half-line. By solving the associated martingale problem and employing weak convergence techniques, we prove that under appropriate scaling, the process converges to the general Brownian motion in the $J_1$-Skorokhod topology. The main novelty of our approach lies in  a result on the asymptotic behavior of the local time of the boundary random walks, allowing us to derive a CLT result for the full spectrum of Feller-type limit processes on the half-line.
\end{abstract}

\maketitle
\allowdisplaybreaks

\section{Introduction}

Random walks form a fundamental class of stochastic processes, with applications in areas such as statistical physics, potential theory, and mathematical finance. One of their central asymptotic results is Donsker's invariance principle, which establishes the convergence of diffusively rescaled random walks to Brownian motion.

Diffusion processes with boundary conditions were first systematically studied by Feller; see \cite{Feller1952,Feller1954}. Feller characterized one-dimensional diffusion processes through the domains of their infinitesimal generators. Later, Itô and McKean \cite{ItoMcKean1963} gave a probabilistic description of these processes in terms of their behavior at the boundary. A detailed account of Brownian motions on the half-line can also be found in Knight \cite{Knight1981}.

Brownian motions on the half-line are characterized by three non-negative coefficients $c_1,c_2,c_3$ satisfying

\begin{equation*}
    c_1 + c_2 + c_3 = 1,
\end{equation*}

which appear in the boundary condition

\begin{equation*}
    c_1f(0) - c_2f'(0) + \frac{c_3}{2}f''(0) = 0.
\end{equation*}

These coefficients describe the different mechanisms acting at the origin and give rise to reflected, absorbed, sticky, elastic, exponential holding, and mixed Brownian motions. The formal choice $c_1=1$ corresponds to the killed Brownian motion. This case is not covered by the usual Feller characterization on the closed half-line and will therefore be treated separately.

\begin{figure}[!htb]
    \centering
    \begin{tikzpicture}[scale=1]
        \draw (1,0)--(8,0);

        \draw (4.5,0.6) node[above]{$\displaystyle\pfrac{1}{2}$};
        \draw (5.5,0.6) node[above]{$\displaystyle\pfrac{1}{2}$};
        \draw (1.5,0.6) node[above]{$\displaystyle\pfrac{B}{N^{\beta}}$};
        \draw (0,0.6) node[above]{$\displaystyle\pfrac{A}{N^{\alpha}}$};
        \draw (1,-1.58) node[above]{$\displaystyle{1 - \pfrac{B}{N^\beta} - \pfrac{A}{N^{\alpha}}}$};

        \centerarc[thick,<-](1.5,0)(30:140:0.6);
        \centerarc[thick,->](0,-0.9)(55:125:1.5);
        \centerarc[thick,->](4.5,0)(40:150:0.6);
        \centerarc[thick,<-](5.5,0)(30:140:0.6);
        \draw[thick,->] (0.8,-0.3) arc (130:400:0.35);

        \draw (-1,-0.2) node[below]{$\Delta$};
        \draw (1,-0.2) node[below]{$0$};
        \draw (2,-0.2) node[below]{$1$};
        \draw (3,-0.2) node[below]{$2$};
        \draw (4,-0.2) node[below]{$3$};
        \draw (5,-0.2) node[below]{$4$};
        \draw (6,-0.2) node[below]{$5$};
        \draw (7,-0.2) node[below]{$6$};
        \draw (8,-0.2) node[below]{$7$};
        \draw (8.7,0) node[left]{...};

        \filldraw[ball color=black] (1,0) circle (.15);
        \filldraw[fill=white] (2,0) circle (.15);
        \filldraw[fill=white] (3,0) circle (.15);
        \filldraw[fill=white] (4,0) circle (.15);
        \filldraw[fill=white] (5,0) circle (.15);
        \filldraw[fill=white] (6,0) circle (.15);
        \filldraw[fill=white] (7,0) circle (.15);
        \filldraw[fill=white] (8,0) circle (.15);
        \filldraw[fill=white] (-1,0) circle (.15);
    \end{tikzpicture}
    \caption{Jump probabilities for the \textit{boundary random walk.}}
    \label{GF1}
\end{figure}
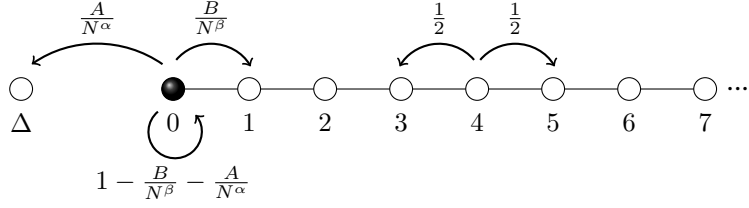

We consider a family of discrete-time random walks on the non-negative integers whose transition probabilities are homogeneous away from the origin. At the origin, the walk may remain at zero, jump to $1$, or be sent to the cemetery state. These probabilities depend on four non-negative parameters $A,B,\alpha$, and $\beta$, as illustrated in Figure~\ref{GF1}. Under diffusive scaling, different choices of these parameters produce different types of Brownian behavior at the boundary.

Scaling limits of discrete processes with non-trivial boundary behavior have been studied in several settings. Harrison and Lemoine \cite{Harrison1981} obtained sticky Brownian motion as a diffusion limit of modified storage processes. Amir \cite{Amir1991} constructed sticky Brownian motion as an almost sure uniform limit of time-changed random walks. Csáki, Földes, and Révész \cite{Csaki2008} established a strong invariance principle relating nearest-neighbor random walks and transient Bessel processes, including estimates on their local times. More recently, Ramírez and Uribe Bravo \cite{Ramirez2024} extended the notion of stickiness to Lévy processes through time-change equations. Related results for non-homogeneous random walks can be found in Menshikov, Popov, and Wade \cite{MenshikovPopovWade2017}.

Further connections between continuous and discrete processes with boundary behavior have recently been obtained by Li \cite{Li2024}, who showed that Feller Brownian motions can be transformed into birth-death processes through a unique time change and, conversely, that the corresponding birth-death processes can be obtained from Feller Brownian motions by time change.

The phase diagram associated with the family of boundary random walks considered here was previously obtained in \cite{Erhard2024} through a Trotter--Kato approach based on convergence of generators. That work establishes functional convergence to the different Brownian motions on the half-line and, in several regimes, provides quantitative estimates for the convergence of their one-dimensional distributions. The purpose of the present paper is to give a direct probabilistic proof of these scaling limits through tightness and the martingale problem.

The main new ingredient is an analysis of the local time of the discrete walk at the origin. More precisely, we derive an upper bound for the expected boundary local time that is uniform when the transition probabilities depend on the scaling parameter. This estimate controls the contribution of the origin in the discrete martingale decomposition and allows us to identify the subsequential limits in the different parameter regimes. The killed Brownian motion regime is excluded from our results.

The paper is organized as follows. In Section~\ref{sec:2}, we introduce the discrete model, recall the Brownian motions on the half-line, and state the main results. In Section~\ref{sec:5}, we establish the estimate on the local time at the origin. In Section~\ref{sec:3}, we prove tightness of the rescaled walks. In Section~\ref{sec:4}, we identify the limit points through the corresponding martingale problems and establish well-posedness. Section~\ref{sec:6} completes the proof of the main theorem.

\section{The framework}\label{sec:2}
\subsection{Notation}
We use the conventions $\bb{N}\;:=\;\{1,2,\ldots\}$ and $\bb{N}_0\;:=\;\{0,1,2,\ldots\}$. For two non-negative sequences $(a_N)_{N\in\bb{N}}$ and $(b_N)_{N\in\bb{N}}$, we write $a_N\;\lesssim\;b_N$, or equivalently $a_N\;=\;O(b_N)$, if there exists a constant $C>0$, independent of $N$, such that $a_N\;\leq\;Cb_N$ for all sufficiently large $N$. We write $a_N\;=\;o(b_N)$ if
\begin{equation*}
    \lim_{N\to\infty}\frac{a_N}{b_N}\;=\;0.
\end{equation*}

Throughout the paper, $\Delta$ denotes the cemetery state. Whenever $\Delta$ is adjoined to a state space, it is understood to be an isolated point. Once a process reaches $\Delta$, it remains there forever. 

\subsection{Discrete model}
We now define a family of discrete-time Markov chains which, under diffusive scaling, approximate a class of one-dimensional Brownian-type processes on the half-line with nontrivial boundary behavior.

Such processes are characterized by their transition probabilities at the origin, which are governed by four non-negative parameters: $\alpha, \beta, A, B$. The parameter $\alpha$ controls the scaling of the probability of jumping from the origin to the cemetery state, while $\beta$ governs the probability of jumping from the origin to the bulk. The parameters $A$ and $B$ represent the intensities of these transitions.

Fix $\alpha, \beta, A, B \geq 0$. We consider integers $N$ sufficiently large so that
\begin{equation*}
    \frac{A}{N^\alpha} \;+\; \frac{B}{N^\beta} \;\leq\; 1.
\end{equation*}
For each such $N\in\bb{N}$, we define the $(\alpha,\beta,A,B)$-boundary random walk $X^N := \{X_n^N:n\in\bb{N}_0\}$ on the state space $\tilde{\bb{G}} := \{\Delta\}\cup\bb{N}_0$. Its discrete generator $\hat{\msf{L}}_N$ acts on bounded functions $f:\tilde{\bb{G}}\to\bb{R}$ via
\begin{equation}\label{eq:discretegenerator}
\hat{\msf{L}}_N\f{x} \;:=\;
\begin{cases}
    \frac{1}{2}\left[\f{x+1} \;+\; \f{x-1} \;-\; 2\f{x}\right]\quad&\text{if } x\in\bb{N},\\
    \frac{A}{N^\alpha}\left[\f{\Delta} \;-\; \f{0}\right] \;+\; \frac{B}{N^\beta}\left[\f{1} \;-\; \f{0}\right]\quad&\text{if } x = 0,\\
    0\quad&\text{if } x = \Delta.
\end{cases}
\end{equation}

Thus, when the walk is at the origin, it jumps to $\Delta$ with probability $p_N(0,\Delta) := A/N^\alpha$, jumps to $1$ with probability $p_N(0,1) := B/N^\beta$, and remains at the origin with probability $p_N(0,0) := 1-A/N^\alpha-B/N^\beta$ while for $x \in \bb{N}$, it jumps to $x+1$ or $x-1$ with equal probability $p_N(x,x+1) = p_N(x,x-1) = 1/2$. See Figure~\ref{GF1} for a diagram illustrating the jump probabilities near the origin. Note that the family of processes $X^N$ is absorbed upon hitting the cemetery state there are no transitions from $\Delta$ back to $\bb{N}_0$. Observe that, as soon as the process returns to the positive axis, it behaves as a homogeneous random walk, which is known to converge to standard Brownian motion.

Denote by $\mc{F}_n^{X^{N}} := \sigma\{X_j^N : j \leq n\}$ the natural filtration generated by the process that encodes all information up to time $n$.

We are interested in understanding the scaling limit of the above discrete process as $N \to \infty$. By tuning the jump probabilities at the origin through the parameters $\alpha, \beta, A, B$ we aim to approximate a family of continuous diffusion processes on the half-line with nontrivial boundary behavior. In particular, our goal is to show that the rescaled boundary random walk $B^N$ defined below converges, in the Skorokhod topology, to a diffusion process which evolves as standard Brownian motion in the interior of $[0,\infty)$, but exhibits a mixture of reflection, absorption, and killing at the origin.

Such diffusion processes have been fully characterized in~\cite{ItoMcKean1963}. The authors showed that all Brownian motions on the half-line with continuous paths and the strong Markov property can be described via a boundary condition involving three non-negative parameters.

Define the rescaled spatial lattice with cemetery state by
\begin{equation*}
    \tilde{\bb{G}}_N \;:=\; \{\Delta\} \;\cup\; \frac{1}{N}\bb{N}_0.
\end{equation*}

In addition, for $T > 0$, we define the rescaled boundary random walk $B^N$ over the state space $\tilde{\bb{G}}_N$ by
\begin{equation*}
    B_t^N \;:=\;
    \begin{cases}
        \dfrac{1}{N}X_{\lfloor N^2t\rfloor}^N\quad&\text{if } X_{\lfloor N^2t\rfloor}^N\in\bb{N}_0,\\
        \Delta\quad&\text{if } X_{\lfloor N^2t\rfloor}^N = \Delta.
    \end{cases}
\end{equation*}
where $t\in[0,T]$, which corresponds to the original boundary random walk $X^N$ rescaled by a factor of $1/N$ in space, and by a factor of $N^2$ in time. In particular, the spatial rescaling is applied only before the process reaches the cemetery state.

This scaling is motivated by the classical diffusive regime, under which symmetric nearest-neighbor random walks converge to Brownian motion. The time acceleration by $N^2$ ensures nontrivial limiting dynamics, while the spatial rescaling by $1/N$ allows us to interpret the process in the continuous state space $[0,\infty)\cup\{\Delta\}$.

We use the convention $\f{\Delta} = 0$ for functions on the extended state space. The discrete generator $\msf{L}_N : \mc{C}_0(\tilde{\bb{G}}_N) \to \mc{C}_0(\tilde{\bb{G}}_N)$ of the $(\alpha, \beta, A, B)$-boundary rescaled random walk $B^N$ is defined via
\begin{equation}\label{eq:boundaryRWgen}
\begin{split}
\msf{L}_N\gn{f}{x}{} \;&=\; \frac{N^2}{2}\left[\gn{f}{x+\pfrac{1}{N}}{} \;+\; \gn{f}{x-\pfrac{1}{N}}{} \;-\; 2\gn{f}{x}{}\right]\one_{[x>0]}\\
&\qquad \;+\; N^2\left[\frac{A}{N^\alpha}\left[\gn{f}{\Delta}{} \;-\; \gn{f}{0}{}\right] \;+\; \frac{B}{N^\beta}\left[\gn{f}{\pfrac{1}{N}}{} \;-\; \gn{f}{0}{}\right]\right]\one_{[x=0]}\\
&=\; \frac{1}{2}\gn{\Delta^Nf}{x}{}\one_{[x>0]} \;+\; \left[\frac{B}{N^{\beta-1}}\nabla^N\f{0} \;-\; \frac{A}{N^{\alpha-2}}\f{0}\right]\one_{[x=0]}.
\end{split}
\end{equation}
At the cemetery state, we set $\msf{L}_N\f{\Delta} = 0$. The discrete Laplacian and the right discrete gradient appearing above are defined by
\begin{equation*}
    \Delta^N\f{x} \;:=\; N^2\left[\f{x+\pfrac{1}{N}} \;+\; \f{x-\pfrac{1}{N}} \;-\; 2\f{x}\right]
\end{equation*}
and
\begin{equation*}
    \nabla^N\f{0} \;:=\; N\left[\f{\pfrac{1}{N}} \;-\; \f{0}\right].
\end{equation*}

It is noteworthy that the term involving $\one_{[x>0]}$ in equation~\eqref{eq:boundaryRWgen} encodes the homogeneous behavior of the random walk away from the origin, which is known to converge to standard Brownian motion. The second term, involving $\one_{[x=0]}$, captures the nontrivial boundary interactions at the origin, which depend on the scaling parameters $A/N^\alpha$ and $B/N^\beta$, where $\alpha, \beta \geq 0$. By a suitable choice of these parameters, we can obtain different limiting behaviors at the boundary, leading to a rich variety of diffusion processes on the half-line.

\subsection{Limiting Diffusions: The general Brownian motion}
The family of diffusion processes that may arise as scaling limits of boundary random walks on the half-line was completely described in ~\cite[Chapter 6]{Knight1981}. These processes have continuous paths up to their lifetime and are strong Markov, coincide with Brownian motion away from the origin, and are distinguished by their boundary behavior at zero.

The key idea is that all such processes coincide with Brownian motion until they hit zero, but may then be reflected, absorbed, or killed--with arbitrary mixtures of these mechanisms allowed. This leads to a rich family of diffusions indexed by a triplet of non-negative coefficients.

By $B_0(t)$, we denote the Brownian motion on $[0,\infty)$ absorbed upon hitting zero. By $\Delta$ we represent the cemetery state, which serves as an isolated point appended to the state space. Let $\tau_0$ denote the hitting time of the origin. Accordingly, we define the extended state space $\bb{G} := \{\Delta\}\cup[0,\infty)$, where the process can be sent to $\Delta$ only from the origin.
\begin{definition}\cite[page 153]{Knight1981}
    A general Brownian motion on the positive half-line is a diffusion process $\msf{W}$ on the set $\bb{G}$ such that the absorbed process $\{\msf{W}(t\wedge \tau_0):t\geq 0\}$ has the same distribution as $B_0$ for any starting point $x\geq 0$.
\end{definition}

Knight's theorem shows that this class of processes is fully characterized by a boundary condition at the origin involving three non-negative parameters. These parameters encode killing, reflection, and holding or absorption at the origin, respectively. In the interior of $[0,\infty)$, the process evolves as standard Brownian motion, while the behavior at the origin is entirely encoded by this boundary condition.

We now introduce the functional setting in which the limiting processes will be defined.

\begin{definition}\rm
The space of continuous functions vanishing at infinity, denoted by $\mc{C}_0(\bb{G})$, consists of the functions $f:\bb{G}\to\bb{R}$ whose restriction to $[0,\infty)$ is continuous and satisfies $\lim_{x\to\infty}f(x) = 0$, with the convention $\f{\Delta} = 0$. We denote by $\mc{C}_0^2(\bb{G})$ the set of functions $f\in\mc{C}_0(\bb{G})$ that are twice continuously differentiable on $(0,\infty)$ and whose second derivative admits a continuous extension to $[0,\infty)$ vanishing at infinity. The values $f'(0)$ and $f''(0)$ are understood as right limits.
\end{definition}

\begin{theorem}\cite[Theorem 6.2, page 157]{Knight1981}\label{thm:Fellerprocess}
Any general Brownian motion $\msf{W}$ on $[0,\infty)$ has generator $\mc{L} = \frac{1}{2}\frac{d^2}{dx^2}$ with corresponding domain
\begin{equation*}
\mf{D}(\mc{L}) \;=\; \left\{f\in\mc{C}_0^2(\bb{G}):f''\in\mc{C}_0([0,\infty),\bb{R})\;\text{ and }\;c_1f(0) \;-\; c_2f'(0) \;+\; \frac{c_3}{2}f''(0) \;=\; 0\right\},
\end{equation*}
for some $c_i \geq 0$ such that $c_1 + c_2 + c_3 = 1$ and $c_1\neq 1$.
\end{theorem}

In this setting, the parameters $c_1$, $c_2$, and $c_3$ lie in the unit simplex, i.e., $c_1 + c_2 + c_3 = 1$ with $c_i \geq 0$, and the case $c_1 = 1$ is excluded. These coefficients appear in the domain of the generator of the process, which is the classical Laplacian on $[0,\infty)$ subject to a generalized boundary condition.

In~\cite{Erhard2024}, the authors rigorously established a functional central limit theorem for a broad class of time-continuous boundary random walks on the rescaled positive integers defined as in~\eqref{eq:boundaryRWgen}, identifying all possible diffusive limits under general scaling regimes. Their results show that the limiting behavior depends critically on the boundary transition rates governed by the parameters $A/N^\alpha$ and $B/N^\beta$.

More precisely, by tuning the parameters $\alpha$ and $\beta$, the rescaled boundary random walk may converge to different Feller diffusions on the extended half-line $[0,\infty)\cup\{\Delta\}$, depending on the asymptotic strength of the transition rates at the origin. These limiting processes are fully characterized by the generator $\mc{L} = \frac{1}{2}\frac{d^2}{dx^2}$ acting on a domain of functions vanishing at infinity and satisfying the boundary condition
\begin{equation*}
    c_1 f(0) \;-\; c_2 f'(0) \;+\; \frac{c_3}{2} f''(0) \;=\; 0,
\end{equation*}
where $c_1,c_2,c_3 \geq 0$ and $c_1 + c_2 + c_3 = 1$, with $c_1\neq 1$. The values of $\alpha$, $\beta$, $A$, and $B$ determine the effective coefficients $c_i$ and hence the nature of the limit process.

A particularly rich case occurs when $\alpha = 2$ and $\beta = 1$, with both $A,B>0$. In this regime, the limiting process is the {\it mixed Brownian motion}, which incorporates the full spectrum of boundary behaviors at the origin. Unlike classical Brownian variants--such as the absorbed, reflected, or sticky Brownian motions--which exhibit a single type of interaction at the boundary, the mixed Brownian motion blends all three mechanisms into a unified process. It evolves as standard Brownian motion in the interior, but at the origin, it exhibits two degrees of freedom: one controls how elastic the particle is before being killed, and the other governs the balance between reflection and absorption.

In particular, for the mixed Brownian motion, none of the coefficients vanish, that is, $c_1,c_2,c_3>0$. So the process truly displays all three boundary behaviors simultaneously. These parameters lie in the unit simplex $c_1 + c_2 + c_3 = 1$, meaning that once any two are chosen, the third is uniquely determined. This structure reflects the fact that the process admits exactly two degrees of freedom at the boundary, encoding both the elasticity of reflection prior to killing and the proportion of time the particle remains absorbed at the origin.

If $\beta = 1$ and $\alpha>2$, the limit is the {\it sticky Brownian motion} characterized by the choice $c_1 = 0$. This process behaves like standard Brownian motion away from the origin, but when it reaches zero, it tends to remain there for a positive amount of Lebesgue time before resuming its motion. The stickiness arises from a nonzero boundary local time density, which slows down the particle and makes the origin ``attractive'' in a temporal sense, without necessarily leading to absorption.

In the regime where $\alpha = \beta + 1$ with $\beta\in[0,1)$, the process converges to the {\it elastic Brownian motion}, corresponding to the case $c_3 = 0$. In this case, the particle is reflected at the origin as in the reflected Brownian motion, but it may also be killed when the local time at the origin reaches an independent exponential random variable. This construction reflects partial permeability of the boundary, interpolating between reflection and killing.

When $\alpha = 2$ and $\beta>1$, the limiting process is given by the {\it exponential holding Brownian motion}, which corresponds to the case $c_2 = 0$: upon hitting the origin, the process becomes trapped there for an exponentially distributed holding time, after which it is sent to the cemetery state. Unlike the sticky Brownian motion, the holding time here is independent of the trajectory and determined by an external exponential clock, representing a delayed killing mechanism.

If $\alpha>\beta+1$ with $\beta\in[0,1)$, the walk converges to the {\it reflected Brownian motion}, in which the origin acts as a perfectly reflecting boundary: the particle bounces instantaneously upon hitting zero, and the process remains supported on $[0,\infty)$. There is no killing or absorption in this case. This behavior corresponds to the case in which the coefficient $c_2 = 1$.

In the case where $\alpha>2$ and $\beta>1$, the process converges to the {\it absorbed Brownian motion}, which means that the particle is trapped as soon as it hits the origin. This behavior corresponds to the coefficient $c_3 = 1$.

Finally, the {\it killed Brownian motion} arises in a degenerate regime where the process is shifted slightly to the right and does not actually reach the origin in the limit. This occurs when $\alpha<\min\{2,\beta+1\}$. In this setting, the process is sent to the cemetery state upon approaching zero, but without spending time there. Formally, this corresponds to the case $c_1 = 1$ in the boundary condition, which falls outside of the standard Feller framework. This regime is not treated in the present paper.

\subsection{Main result}

Theorem~\ref{thm:main} establishes the convergence of a suitably rescaled boundary random walk $B^N$ to the class of Feller diffusions on the half-line, as classified by Knight's Theorem~\ref{thm:Fellerprocess}, which includes the wide spectrum of processes: the classical reflected and absorbed Brownian motions; the more delicate sticky and elastic Brownian motions; the Brownian motion with exponential holding at the origin; and the mixed Brownian motion, which simultaneously displays reflection, absorption, and killing. Each of these limiting behaviors emerges from specific choices of the boundary parameters $A, B, \alpha, \beta$, which govern the asymptotic transition probabilities from the origin to the interior or to the cemetery state.

In this sense, the boundary random walk serves as a discrete and flexible approximation scheme that unifies all such processes within a single framework.

\begin{theorem}[Scaling limit to the general Brownian motion]\label{thm:main}

Let $\bb{G} := [0,\infty)\cup\{\Delta\}$ and consider the sequence of rescaled processes $\left\{B_t^N:t\in[0,T]\right\}_{N\in\bb{N}}$ whose generator is given by $\msf{L}_N$ as in~\eqref{eq:boundaryRWgen}. Fix $A,B>0$ and $\alpha,\beta\geq 0$. Assume that $B_0^N$ converges in distribution to a $\bb{G}$-valued random variable $\xi$. In each of the cases below, the limiting process starts with the distribution of $\xi$.
\begin{enumerate}
\item Assume $\alpha = 2$ and $\beta = 1$, and set
\begin{equation*}
    c_1 \;=\; \frac{A}{1+A+B},\qquad c_2 \;=\; \frac{B}{1+A+B},\qquad c_3 \;=\; \frac{1}{1+A+B}.
\end{equation*}
Then, as $N\to\infty$, the sequence of rescaled processes $\left\{B_t^N:t\in[0,T]\right\}_{N\in\bb{N}}$ converges in distribution in the Skorokhod space $\mc{D}([0,T],\bb{G})$ endowed with the $J_1$-topology, to the mixed Brownian motion of parameters $(c_1,c_2,c_3)$.

\item Assume $\alpha>2$ and $\beta = 1$, and set
\begin{equation*}
    c_1 \;=\; 0,\qquad c_2 \;=\; \frac{B}{1+B},\qquad c_3 \;=\; \frac{1}{1+B}.
\end{equation*}
Then, as $N\to\infty$, the sequence of rescaled processes $\left\{B_t^N:t\in[0,T]\right\}_{N\in\bb{N}}$ converges in distribution in the Skorokhod space $\mc{D}([0,T],\bb{G})$ endowed with the $J_1$-topology, to the sticky Brownian motion of parameters $(c_1,c_2,c_3)$.

\item Assume $\alpha = 2$ and $\beta>1$, and set
\begin{equation*}
    c_1 \;=\; \frac{A}{1+A},\qquad c_2 \;=\; 0,\qquad c_3 \;=\; \frac{1}{1+A}.
\end{equation*}
Then, as $N\to\infty$, the sequence of rescaled processes $\left\{B_t^N:t\in[0,T]\right\}_{N\in\bb{N}}$ converges in distribution in the Skorokhod space $\mc{D}([0,T],\bb{G})$ endowed with the $J_1$-topology, to the exponential holding Brownian motion of parameters $(c_1,c_2,c_3)$.

\item Assume $\alpha>2$ and $\beta>1$, and set $c_1 = c_2 = 0$ and $c_3 = 1$. Then, as $N\to\infty$, the sequence of rescaled processes $\left\{B_t^N:t\in[0,T]\right\}_{N\in\bb{N}}$ converges in distribution in the Skorokhod space $\mc{D}([0,T],\bb{G})$ endowed with the $J_1$-topology, to the absorbed Brownian motion.

\item Assume $\beta\in[0,1)$ and $\alpha>\beta+1$, and set $c_1 = c_3 = 0$ and $c_2 = 1$. Then, as $N\to\infty$, the sequence of rescaled processes $\left\{B_t^N:t\in[0,T]\right\}_{N\in\bb{N}}$ converges in distribution in the Skorokhod space $\mc{D}([0,T],\bb{G})$ endowed with the $J_1$-topology, to the reflected Brownian motion.

\item Assume $\beta\in[0,1)$ and $\alpha = \beta+1$, and set
\begin{equation*}
    c_1 \;=\; \frac{A}{A+B},\qquad c_2 \;=\; \frac{B}{A+B},\qquad c_3 \;=\; 0.
\end{equation*}
Then, as $N\to\infty$, the sequence of rescaled processes $\left\{B_t^N:t\in[0,T]\right\}_{N\in\bb{N}}$ converges in distribution in the Skorokhod space $\mc{D}([0,T],\bb{G})$ endowed with the $J_1$-topology, to the elastic Brownian motion of parameters $(c_1,c_2,c_3)$.
\end{enumerate}
\end{theorem}

The killed Brownian motion corresponds formally to the excluded case $c_1 = 1$. Since it is not a Feller process on the closed half-line $\bb{G} = [0,\infty)\cup\{\Delta\}$ under the formulation of Theorem~\ref{thm:Fellerprocess}, we will not include it.

\section{Local time of boundary random walks at the origin}\label{sec:5}

This section is devoted to the asymptotic analysis of the local time at the origin for the boundary random walk $X^N$ defined in Section~\ref{sec:2}. Briefly, the local time at the origin, denoted by $L^N_t$, quantifies how much time the random walk spends at zero up to time $t$.

\subsection{Asymptotic Analysis of Return Probabilities and Expected Local Time}

Let $p \in (0,1)$, and denote by $X^p$ (or simply $X$) a boundary random walk on $\bb{N}_0$ with homogeneous transition probabilities away from the origin and such that $p(0,1) = p = 1-p(0,0)$ at the origin.

Define $F_k := P(X_k = 0)$ and assume $X_0 = 0$. Our goal is to derive an estimate for $F_k$ by employing a renewal-type recurrence relation, which is a fundamental technique in the analysis of random walks. We start by observing the following two mutually exclusive scenarios for the random walk after $k$ steps:

\begin{enumerate}

\item The random walk does not leave the origin at the first step, which occurs with probability $1-p$. After this step, its behavior is governed by $F_{k-1}$;

\item The random walk moves from $X_0 = 0$ to $X_1 = 1$. In this case, for $1\leq j\leq\lfloor k/2\rfloor$, we consider the event that its first return to $0$ occurs at time $2j$:

\begin{itemize}

\item At the first step, it moves from $0$ to $1$ with probability $p$.

\item The next $2j-2$ steps describe an excursion starting and ending at $1$ without visiting zero.

\item At time $2j$, the walk returns from $1$ to $0$ with probability $1/2$.

\item During the remaining $k - 2j$ steps, the process restarts from the origin, and its behavior is governed by $F_{k-2j}$.

\end{itemize}

\end{enumerate}

If $h_j$ denotes the probability of the first return to $0$ at time $2j$, then

\begin{equation*}
    h_j \;=\; P(\text{first return to } 0 \text{ at time } 2j) \;=\; \frac{p}{2} \;\cdot\; \frac{C_{j-1}}{2^{2j - 2}}.
\end{equation*}

Thus, the recursive equation for $F_k$ is given by

\begin{equation*}
\begin{cases}
    F_k \;=\; (1 \;- \; p)F_{k-1} \;+\; \displaystyle\sum_{j = 1}^{\lfloor k/2 \rfloor} h_jF_{k \;- \; 2j},\\
    F_0 \;=\; 1.
\end{cases}
\end{equation*}

The constant $C_{j-1}$ counts the number of paths that start at $1$, remain strictly positive during the following $2j-2$ steps, and return to $1$ at the end of the excursion. These are given by the {\it Catalan numbers}, whose exact form is

\begin{equation*}
    C_n \;=\; \frac{1}{n + 1}\binom{2n}{n}.
\end{equation*}

Using $F_0=1$ and the expression for $h_j$, we obtain

\begin{equation}\label{eq:renewalRelation}
    F_k \;=\; (1 \;-\; p)F_{k-1} \;+\; \sum_{j = 1}^{\lfloor k/2 \rfloor} \frac{p}{2} \;\cdot\; \binom{2j - 2}{j - 1}\frac{1}{j \;\cdot\; 2^{2j \;-\; 2}}F_{k - 2j}.
\end{equation}

To solve the recurrence relation for $F_k$, we employ the method of generating functions. This powerful analytical technique, known alternatively as {\it Z-transforms} in discrete mathematics or {\it Green's functions} in the study of Markov chains and renewal processes, converts the convolution sum into an algebraic product, thereby substantially simplifying the analysis.

The generating function for the sequence ${F_k}$ is defined by

\begin{equation*}
    F(x) \;:=\; \sum_{k \;=\; 0}^{\infty}F_kx^k,
\end{equation*}

and the renewal relation translates into an algebraic equation in terms of generating functions:

\begin{equation}\label{eq:renewalRelation1}
    F(x) \;=\; (1 - p)xF(x) \;+\; H(x^2)F(x) \;+\; 1,
\end{equation}

where $H(x)$ is the generating function of the first-return probabilities $h_j$. see~\cite{GF}. That is,

\begin{equation*}
    H(x) \;=\; \sum_{j \;=\; 1}^{\infty}h_jx^j.
\end{equation*}

Let us now focus our analysis on the second term in equation~\eqref{eq:renewalRelation1}. The generating function $H(x)$ can be expressed as follows:

\begin{equation}\label{eq:renewalRelation3}
\begin{split}
    H(x) \;&=\; \frac{p}{2}\sum_{j = 1}^{\infty}\frac{C_{j-1}}{4^{j-1}}x^j\\
    &\;=\; \frac{px}{2}\sum_{n = 0}^{\infty}C_n\Bigl(\frac{x}{4}\Bigr)^n\\
    &\;=\; \frac{px}{2}\frac{2(1 \;-\; \sqrt{1 \;-\; x})}{x}\\
    &\;=\; p\bigl(1 \;-\; \sqrt{1 \;-\; x}\bigr).
\end{split}
\end{equation}

In the third equality, we used the generating function of the Catalan numbers, which is valid for $|x|<1$. From relation~\eqref{eq:renewalRelation1} and by plugging equation~\eqref{eq:renewalRelation3}, one can show that $F$ satisfies

\begin{equation}\label{eq:renewalRelation2}
    F(x) \;=\; \frac{1}{1 \;-\; (1 \;-\; p)x \;-\; H(x^2)}.
\end{equation}

The long-term behavior of $F_k$ can be determined by analyzing the singularities of its generating function $F(x)$. This approach uses singularity analysis, which establishes a link between

\begin{itemize}

\item The local behavior of $F(x)$ near its dominant singularity;

\item The asymptotic decay of its coefficients $F_k$ as $k$ tends to infinity.

\end{itemize}

We focus our analysis on the local expansion near the dominant singularity at $x=1$. Introducing the parameter $\eps = 1-x$, we consider the limit as $\eps\to0^+$. Observe that

\begin{equation}\label{eq:renewalRelation4}
\begin{split}
    1 \;-\; (1 \;-\; p)x \;-\; H(x^2) \;&=\; 1 \;-\; (1 \;-\; p)x \;-\; p\bigl(1 \;-\; \sqrt{1 \;-\; x^2}\bigr)\\
    &\;=\; (1 \;-\; p)(1 \;-\; x) \;+\; p\sqrt{1 \;-\; x^2}\\
    &\;=\; (1 \;-\; p)\eps \;+\; p\sqrt{2}\eps^{1/2}\sqrt{1 \;-\; \pfrac{\eps}{2}}\\
    &\;=\; (1 \;-\; p)\eps \;+\; p\sqrt{2}\eps^{1/2}\bigl(1 \;-\; \pfrac{\eps}{4} \;+\: O(\eps^2)\bigr)\\
    &\;=\; p\sqrt{2}\eps^{1/2}\Bigl[1 \;+\: \frac{1 \;-\: p}{p\sqrt{2}}\eps^{1/2} \;-\: \frac{\eps}{4} \;+\: O(\eps^2)\Bigr].
\end{split}
\end{equation}

Therefore, plugging equation~\eqref{eq:renewalRelation4} into~\eqref{eq:renewalRelation2}, we obtain

\begin{equation*}
    F(x) \;=\; \frac{\eps^{-1/2}}{p\sqrt{2}}\left[1 \;-\; \frac{1 \;-\; p}{p\sqrt{2}}\eps^{1/2} \;+\; \left(\frac{(1 \;-\; p)^2}{2p^2} \;+\; \frac{1}{4}\right)\eps \;+\; O(\eps^{3/2})\right].
\end{equation*}
The dominant non-analytic term in the expansion at $x=1$ is proportional to $(1-x)^{-1/2}$. Applying the transfer theorem in~\cite[Corollary 2]{F-O} to this dominant singular term, we obtain

\begin{equation*}
    F_k \;=\; [x^k]F(x) \;\sim\; \frac{A_1}{\Gamma\bigl(\tfrac{1}{2}\bigr)}k^{-1/2},
\end{equation*}

where

\begin{equation*}
    A_1 \;=\; \frac{1}{p\sqrt{2}}\quad\text{and}\quad\Gamma\bigl(\tfrac{1}{2}\bigr) \;=\; \sqrt{\pi},
\end{equation*}

leading to the asymptotic behavior of $F_k$ given by

\begin{equation}\label{eq:asymptoticFk}
    F_k \;=\; [x^k]F(x) \;\sim\; \frac{1}{p\sqrt{2\pi}}k^{-1/2}.
\end{equation}

Let $\mf{L}_N^{X^p}(x)$ to denote the local time of $x$ of the random walk $X^p$ after $N$ steps. Formally
\begin{equation*}
    \mf{L}_N^{X^p}(x)\,:=\,\sum_{k=0}^{N}\one_{[X_k^p = x]}\,.
\end{equation*}
The next result, we estimate the expected local time at the origin for the boundary random walk $X^p$. Indeed, by using the asymptotic behavior of $F_k$, we can derive an upper bound for the expected local time at the origin.

For fixed $p\in(0,1)$, equation~\eqref{eq:asymptoticFk} implies that there exists an integer $k_0(p)\geq2$ such that $F_k\leq 2k^{-1/2}/(p\sqrt{2\pi})$ for every $k\geq k_0(p)$. Define
\begin{equation*}
    D_p \;:=\; \max\left\{2, \; p\sqrt{2\pi}\max_{1\leq k<k_0(p)}\sqrt{k}F_k\right\}.
\end{equation*}

Then $D_p<\infty$, and 
\begin{equation*}
    F_k \;\leq\; \frac{D_p}{p\sqrt{2\pi}}k^{-1/2}, \qquad k\geq1.
\end{equation*}
Write $a_j=C_j/4^j$, where $C_j$ denotes the Catalan number defined above. We start by expressing the expected local time as
\begin{equation}\label{eq:localTimeTheory1}
\begin{split}
    \Expec{\mf{L}_N^{X^p}(0)}{}{} \;&=\; \sum_{k=0}^{N} P(X_{k}^p=0) \;=\; 1 \;+\; \sum_{k=1}^NF_k\\
    &\overset{\eqref{eq:renewalRelation}}{=}\; 1 \;+\; (1-p)\sum_{k = 1}^NF_{k-1} \;+\; \sum_{k=1}^{N} \sum_{j=1}^{\lfloor k/2\rfloor} \frac{p}{2}\binom{2j-2}{j-1}\frac{1}{j2^{2j-2}} F_{k-2j}\\
    &=\; 1 \;+\; (1-p)\sum_{k=1}^NF_{k-1} \;+\; \frac{p}{2}\sum_{k=1}^N \sum_{j=1}^{\lfloor k/2\rfloor} a_{j-1}F_{k-2j}.
\end{split}
\end{equation}
Now, observe that, since the sum above is finite and $F_0=1$, it follows that
\begin{equation}\label{eq:localTimeTheory1.1}
    \sum_{k=1}^N\sum_{j=1}^{\lfloor k/2\rfloor}a_{j-1}F_{k-2j} \;=\; \sum_{j=1}^{\lfloor N/2\rfloor}a_{j-1}\sum_{r=0}^{N-2j}F_r \;=\; \sum_{j=1}^{\lfloor N/2\rfloor}a_{j-1}\Big(1 \;+\; \sum_{r=1}^{N-2j}F_r\Big),
\end{equation}
where $r=k-2j$, and empty sums are understood to be zero. Recalling that, for each $j \geq 1$, we have
\begin{equation*}
    a_{j-1} \;=\; \frac{1}{j2^{2j-2}}\binom{2j-2}{j-1},
\end{equation*}
one can verify, using the generating function for the Catalan numbers, that
\begin{equation*}
    \sum_{j=1}^\infty a_{j-1} \;=\; 2.
\end{equation*}
Additionally, using the bound defining $D_p$ and comparing the sum with an integral, namely 
\begin{equation*}
    \sum_{r=1}^m r^{-1/2} \;\leq\; \int_0^m x^{-1/2}\,dx \;=\; 2\sqrt{m}
\end{equation*}
for $m\geq1$, it follows that
\begin{equation}\label{eq:localTimeTheory2.1}
\begin{split}
    \sum_{k=1}^N \sum_{j=1}^{\lfloor k/2\rfloor} a_{j-1}F_{k-2j} \;&=\; \sum_{j=1}^{\lfloor N/2\rfloor}a_{j-1}\Big(F_0 \;+\; \sum_{r=1}^{N-2j}F_r\Big)\\
    &\leq\; \sum_{j=1}^{\lfloor N/2\rfloor}a_{j-1}\Big(1 \;+\; \frac{D_p}{p\sqrt{2\pi}}\sum_{r=1}^{N-2j}r^{-1/2}\Big)\\
    &\leq\; \sum_{j=1}^{\lfloor N/2\rfloor}a_{j-1}\Big(1 \;+\; \frac{2D_p}{p\sqrt{2\pi}}\sqrt{N-2j}\Big)\\
    &\leq\; \sum_{j=1}^\infty a_{j-1}\Big(1 \;+\; \frac{2D_p\sqrt{N}}{p\sqrt{2\pi}}\Big) \;=\; 2\Big(1 \;+\; \frac{2D_p\sqrt{N}}{p\sqrt{2\pi}}\Big).
\end{split}
\end{equation}
Similarly, for the second term on the right-hand side of equation~\eqref{eq:localTimeTheory1}, one can show that
\begin{equation}\label{eq:localTimeTheory2.2}
    \sum_{k=1}^N F_{k-1} \;=\; 1 \;+\; \sum_{r=1}^{N-1}F_r \;\leq\; 1 \;+\; \frac{2D_p\sqrt{N}}{p\sqrt{2\pi}}.
\end{equation}
By replacing \eqref{eq:localTimeTheory2.1} and \eqref{eq:localTimeTheory2.2} into \eqref{eq:localTimeTheory1}, we obtain that
\begin{equation}\label{GeneralBound}
    \Expec{\mf{L}_N^{X^p}(0)}{}{} \;\leq\; 2 \;+\; \frac{1-p}{p\sqrt{2\pi}}2D_p\sqrt{N} \;+\; \frac{2D_p}{\sqrt{2\pi}}\sqrt{N}.
\end{equation}

Thus, we have established, partially, the following theorem, which provides an upper bound on the expected local time at the origin for the boundary random walk $X^p$.
\begin{theorem}\label{thm:localTime}
Let $X^p$ be the boundary random walk defined previously, started at zero, with $p\in(0,1)$ fixed. Then, for any $N \in \bb{N}$, the expected local time at the origin satisfies
\begin{equation*}
    \Expec{\mf{L}_N^{X^p}(0)}{}{} \;\leq\; 2 \;+\; \frac{1-p}{p\sqrt{2\pi}}2D_pN^{1/2} \;+\; \frac{2D_p}{\sqrt{2\pi}}N^{1/2}.
\end{equation*}
Moreover, as $N\to \infty$
\begin{equation*}
    \bb{E}\big[\mf{L}_N^{X^p}(0)\big] \;\sim\; \frac{\sqrt{2/\pi}}{p}\sqrt{N}
\end{equation*}
\end{theorem}
\begin{proof}
It only remains to prove the asymptotic behavior. Write $c_p := (p\sqrt{2\pi})^{-1}$. By the asymptotic behavior of $F_k$ in~\eqref{eq:asymptoticFk}, for every $\eps \in (0,1)$ there exists an integer $k_0 := k_0(p,\eps)$ such that
\begin{equation*}
    (1-\eps)c_p k^{-1/2} \;\leq\; F_k \;\leq\; (1+\eps)c_p k^{-1/2}, \qquad k \geq k_0.
\end{equation*}
Therefore, for every $N > k_0$, we have
\begin{equation}\label{eq:asymptoticFksum1}
    (1 - \eps)c_p\sum_{k=k_0}^Nk^{-1/2} \;\leq\; \sum_{k=1}^N F_k \;-\; \sum_{k=1}^{k_0-1}F_k \;\leq\; (1 + \eps)c_p\sum_{k=k_0}^N k^{-1/2}.
\end{equation}
In particular, dividing by $\sqrt{N}$ and taking the limit as $N \to \infty$, we obtain
\begin{equation*}
    (1 - \eps)c_p\lim_{N \to \infty}\frac{1}{\sqrt{N}}\sum_{k=k_0}^N k^{-1/2} \;\leq\; \lim_{N \to \infty}\frac{1}{\sqrt{N}}\sum_{k=1}^N F_k \;\leq\; (1 + \eps)c_p\lim_{N \to \infty}\frac{1}{\sqrt{N}}\sum_{k=k_0}^N k^{-1/2}.
\end{equation*}
Observe that, in \eqref{eq:asymptoticFksum1}, the sums $\sum_{k=1}^{k_0-1}F_k$ are finite and do not contribute to the limit. Finally, using the fact that
\begin{equation*}
    \lim_{N \to \infty}\frac{1}{\sqrt{N}}\sum_{k=k_0}^N k^{-1/2} \;=\; \lim_{N \to \infty}\frac{1}{\sqrt{N}}\sum_{k=1}^N k^{-1/2} \;=\; 2,
\end{equation*}
the proof is complete by taking $\eps\downarrow 0$.
\end{proof}

The next result is a well-known property of the return probabilities of a reversible Markov chain. It is included here for completeness.

\begin{lemma}\label{lem:FkDecreasing}
The sequence of return probabilities $\{F_{2k}\}_{k\geq0}$ is non-increasing.
\end{lemma}
\begin{proof}
Let $P(\cdot,\cdot)$ be the transition matrix of $X^p$. Since $X^p$ is a reversible Markov chain, there exists a stationary distribution, say $\pi$, satisfying the detailed balance condition 
\begin{equation*}
    \pi(i)P(i,j) \;=\; \pi(j)P(j,i), \qquad i,j\in\bb{N}_0.
\end{equation*} 
Now, define $u_n(y) \;=\; \bb{P}_y(X_n^p=0)$, where $\bb{P}_y$ denotes the law of the Markov chain $X^p$ starting at $y$. Also, observe that $u_{n+1}(x) = \sum_y P(x,y)u_n(y)$. Now, by reversibility of $X^p$, and since $X_0^p = 1$, it follows that
\begin{equation}\label{eq:detailed1}
    \pi(x)P(x,0) \;=\; \pi(0)P(0,x) \;=\; P(0,x)\,,
\end{equation}
and
\begin{equation}\label{eq:detailed2}
    \sum_x \pi(x)P(x,y) \;=\; \sum_x \pi(y)P(y,x) \;=\; \pi(y)\,,
\end{equation}
On the other hand, from the above equation and the fact that $\pi(0) = 1$, we have
\begin{equation}\label{eq:detailed3}
    F_{2n} \;=\; \sum_x P^n(0,x)P^n(x,0) \;\overset{\eqref{eq:detailed1}}{=}\; \sum_x \pi(x)P^n(x,0)P^n(x,0) \;=\; \sum_x \pi(x) u_n(x)^2\,.
\end{equation}
Applying the Cauchy-Schwarz inequality, we obtain
\begin{equation*}
\begin{split}
    F_{2n + 2} \;=\; \sum_x \pi(x)u_{n+1}(x)^2 \;&=\; \sum_x \pi(x)\left(\sum_y P(x,y)u_n(y)\right)^2 \\
    &\leq\; \sum_x \pi(x)\sum_y P(x,y)\sum_y P(x,y)u_n(y)^2 \\
    &=\; \sum_y\sum_x \pi(x)P(x,y)u_n(y)^2 \;\overset{\eqref{eq:detailed2}}{=}\; \sum_y \pi(y)u_n(y)^2 \;=\; F_{2n}.
\end{split}
\end{equation*}
concluding the proof.
\end{proof}

\begin{lemma}\label{lem:uniformReturnBound}
The constants $D_p$ in Theorem~\ref{thm:localTime} can be chosen so that
\begin{equation}\label{eq:uniformReturnBound}
    D_* \;:=\; \sup_{p\in(0,1)}D_p \;<\; \infty.
\end{equation}
\end{lemma}

\begin{proof}
From Lemma~\ref{lem:FkDecreasing}, the sequence $\{F_{2k}\}_{k\geq0}$ is non-increasing, meaning that
\begin{equation}\label{eq:uniformReturnBound0}
    (m+1)F_{2m} \;\leq\; \sum_{k=0}^{2m}F_k\,.
\end{equation}
By the condition $F_0 = 1$, we have that
\begin{equation}\label{eq:uniformReturnBound1}
\begin{split}
    \bb{E}[X_{m+1}] \;&=\; \bb{E}[X_{m+1} \;-\; X_0] \;=\; \sum_{j=1}^{m+1} \bb{E}[X_j \;-\; X_{j-1}] \\
    &=\; \sum_{j=1}^{m+1} \bb{E}\big[\bb{E}[X_j \;-\; X_{j-1}|X_{j-1}]\big] \;=\; \sum_{j=1}^{m+1} \bb{E}{\mc{L}g(X_{j-1})}
\end{split}
\end{equation}
where $g(x) = x$ and $\mc{L}$ is the generator of the Markov chain $X^p$. Since $X^p$ is a boundary random walk, we have that $\mc{L}g(0) = p$ and $\mc{L}g(x) = 0$ for $x > 0$. Therefore, we obtain
\begin{equation}\label{eq:uniformReturnBound2}
    \bb{E}[X_{m+1}] \;=\; \sum_{j=1}^{m+1} p \bb{P}(X_{j-1}=0) \;=\; p \sum_{j=0}^{m} F_j \;=\; p \bb{E}[\mf{L}_m^{X^p}(0)]\,.
\end{equation}
To estimate the second moment, observe that, at a positive integer $x$, the conditional mean increment of the squared position is
\begin{equation*}
    \frac{1}{2}\bigl((x+1)^2-x^2\bigr) \;+\; \frac{1}{2}\bigl((x-1)^2-x^2\bigr) \;=\; 1.
\end{equation*}
At zero, this conditional mean increment equals $p$. Therefore,
\begin{equation*}
    \bb{E}[X_{j+1}^2-X_j^2] \;=\; \bb{P}(X_j>0) \;+\; p\bb{P}(X_j=0) \;=\; 1 \;-\; (1-p)F_j \;\leq\; 1.
\end{equation*}
Now, by Cauchy-Schwarz inequality, it follows that
\begin{equation}\label{eq:uniformReturnBound3}
\begin{split}
    \bb{E}[X_{m+1}] \;&\leq\; \sqrt{\bb{E}[X_{m+1}^2]} \;=\; \sqrt{\sum_{j=0}^{m} \bb{E}[X_{j+1}^2 \;-\; X_j^2]} \\
    &\leq\; \sqrt{m+1}.
\end{split}
\end{equation}
Hence, combining equations~\eqref{eq:uniformReturnBound1}, \eqref{eq:uniformReturnBound2}, and \eqref{eq:uniformReturnBound3}, we obtain
\begin{equation*}
    (m+1)F_{2m}  \;\overset{\eqref{eq:uniformReturnBound0}}{\leq}\; \bb{E}[\mf{L}_{2m}^{X^p}(0)] \;\leq\; \frac{\sqrt{2m+1}}{p}.
\end{equation*}
Finally, 
\begin{equation*}
    D_p \;\leq\; \sqrt{2\pi}\sup_{p \in (0,1)}p\max_{1 \leq k \leq k_0(p)}\sqrt{2k}F_{2k} \;\leq\; \sqrt{2\pi}\sup_{p \in (0,1)}\max_{1 \leq k \leq k_0(p)}\frac{\sqrt{2k(2k + 1)}}{k+1} \;=:\; D_* \;<\; \infty,
\end{equation*}
where $D_*$ does not depend on $p$, and then, we can conclude.
\end{proof}

\subsection{A uniform bound for the boundary local time}

We first extend the local-time estimate of Theorem~\ref{thm:localTime} to a boundary random walk with killing at the origin.

For $0<p<1$ and $0\leq q\leq1-p$, let $X^{p,q}$ be the boundary random walk which, at zero, jumps to $1$ with probability $p$, jumps to the absorbing cemetery state $\Delta$ with probability $q$, and remains at zero with probability $1-p-q$. Away from zero, its transition probabilities are the same as those of $X^p$. Thus, $X^{p,0}$ coincides with $X^p$. Both walks are started at zero, and we write
\begin{equation*}
    F_k(p,q) \;:=\; \bb{P}_0(X_k^{p,q}=0), \qquad F_k(p,0) \;=\; F_k(p).
\end{equation*}

\begin{lemma}\label{lem:killingComparison}
For every $0<p<1$, $0\leq q\leq1-p$, and every integer $k\geq0$,
\begin{equation}\label{eq:killingComparison}
    F_k(p,q) \;\leq\; F_k(p).
\end{equation}
Consequently, for every integer $m\geq0$,
\begin{equation}\label{eq:killingLocalTimeComparison}
    \Expec{\mf{L}_m^{X^{p,q}}(0)}{0}{} \;\leq\; \Expec{\mf{L}_m^{X^p}(0)}{0}{}.
\end{equation}
\end{lemma}

\begin{proof}
Fix $p\in(0,1)$ and let $P_q$ and $P_0$ denote the transition operators of $X^{p,q}$ and $X^{p,0}$, respectively. We regard both operators as acting on functions on $\bb{N}_0\cup\{\Delta\}$, with $\Delta$ absorbing in both cases. For any bounded nonnegative function $g$ satisfying $g(\Delta)=0$,
\begin{equation*}
    (P_qg)(0) \;=\; pg(1) \;+\; (1-p-q)g(0) \;\leq\; pg(1) \;+\; (1-p)g(0) \;=\; (P_0g)(0).
\end{equation*}
The two operators coincide at every positive integer, and both give zero at $\Delta$. Thus, $P_qg\leq P_0g$. Since transition operators preserve pointwise inequalities, and $P_0^kg$ is nonnegative and vanishes at $\Delta$, induction gives
\begin{equation*}
    P_q^{k+1}g \;=\; P_q(P_q^kg) \;\leq\; P_q(P_0^kg) \;\leq\; P_0(P_0^kg) \;=\; P_0^{k+1}g.
\end{equation*}
Taking $g=\one_{\{0\}}$, we obtain
\begin{equation*}
    F_k(p,q) \;=\; (P_q^kg)(0) \;\leq\; (P_0^kg)(0) \;=\; F_k(p).
\end{equation*}
Summing over $k=0,\ldots,m$ proves~\eqref{eq:killingLocalTimeComparison}.
\end{proof}
By Lemma~\ref{lem:killingComparison} and Theorem~\ref{thm:localTime},
\begin{equation*}
    \Expec{\mf{L}_m^{X^{p,q}}(0)}{0}{} \;\leq\; 2 \;+\; \frac{2D_p}{p\sqrt{2\pi}}\sqrt{m}.
\end{equation*}
The same constant $D_p$ therefore applies to every admissible killing probability $q$.

We now apply the preceding estimates to the boundary random walk $X^N$ defined in Section~\ref{sec:2}, whose transition probabilities at the origin depend on $N$. We write
\begin{equation*}
    p_N \;:=\; \frac{B}{N^\beta}, \qquad p_{\Delta,N} \;:=\; \frac{A}{N^\alpha},
\end{equation*}
so that the probability of remaining at zero is
\begin{equation*}
    1 \;-\; p_N \;-\; p_{\Delta,N} \;=\; 1 \;-\; \frac{B}{N^\beta} \;-\; \frac{A}{N^\alpha}.
\end{equation*}
We assume $B>0$ and $A,\alpha,\beta\geq0$, and that there exists an integer $N_0\geq1$ such that $0<p_N<1$ and $p_N+p_{\Delta,N}\leq1$ for every $N\geq N_0$. For each fixed $N$, the transition probabilities remain constant in time.

By Lemma~\ref{lem:uniformReturnBound}, we have $D_{p_N}\leq D_*$ for every $N\geq N_0$. Applying Theorem~\ref{thm:localTime} with $p=p_N$ and using Lemma~\ref{lem:killingComparison} with $q=p_{\Delta,N}$ therefore yields
\begin{equation*}
    \Expec{\mf{L}_m^{X^N}(0)}{0}{} \;\leq\; 2 \;+\; \frac{2D_*}{p_N\sqrt{2\pi}}\sqrt{m} \;=\; 2 \;+\; \frac{2D_*}{B\sqrt{2\pi}}N^\beta\sqrt{m}.
\end{equation*}
Here $m$ denotes the number of steps, whereas $N$ is the scaling parameter.

\begin{corollary}\label{cor:localTimeBound}
Under the preceding assumptions, for any initial distribution and every integer $m\geq0$,
\begin{equation}\label{eq:localTimeBound}
    \Expec{\mf{L}_m^{X^N}(0)}{}{} \;\leq\; 2 \;+\; \frac{2D_*}{B\sqrt{2\pi}}N^\beta\sqrt{m}.
\end{equation}
In particular,
\begin{equation*}
    \Expec{\mf{L}_N^{X^N}(0)}{}{} \;\leq\; 2 \;+\; \frac{2D_*}{B\sqrt{2\pi}}N^{1/2+\beta}.
\end{equation*}
Moreover, for every fixed $T>0$,
\begin{equation}\label{eq:localTimeDiffusiveBound}
    \sup_{0\leq t\leq T}\Expec{\mf{L}_{\lfloor tN^2\rfloor}^{X^N}(0)}{}{} \;\leq\; 2 \;+\; \frac{2D_*\sqrt{T}}{B\sqrt{2\pi}}N^{\beta+1}.
\end{equation}
\end{corollary}

\begin{proof}
Fix $N\ge N_0$ and let $\tau_0:=\inf\{k\ge0:X_k^N=0\}$. By the strong Markov property and the monotonicity of the local time, for every initial state $x$,
\begin{equation*}
    \mathbb{E}_x\left[\mf{L}_m^{X^N}(0)\right]
    \;=\;
    \sum_{k=0}^{m}\mathbb{P}_x(\tau_0=k)
    \mathbb{E}_0\left[\mf{L}_m^{X^N}(0)\right]
    \;\le\;
    \mathbb{E}_0\left[\mf{L}_m^{X^N}(0)\right].
\end{equation*}
Averaging over the initial distribution and applying the
preceding estimate proves the first inequality. The remaining inequalities follow by taking $m=N$ and $m=\lfloor tN^2\rfloor$, respectively, and using
$\sqrt{\lfloor tN^2\rfloor}\le N\sqrt T$ for $0\le t\le T$.
\end{proof}

This result states that the expected local time at the origin grows at most on the order of $N^{1/2+\beta}$, where $\beta$ is a parameter that influences the transition probabilities at the boundary. The estimate~\eqref{eq:localTimeDiffusiveBound} provides the corresponding control up to a diffusive time, with a constant independent of $N$.

\begin{lemma}\label{lem:localTimeSecondMoment}
For $m\in\bb{N}$, let $L_m^N(0):=\sum_{j=0}^{m-1}\one_{[X_j^N=0]}$. Then, for every initial state $x\in\bb{N}_0\cup\{\Delta\}$,
\begin{equation}\label{eq:localTimeSecondMoment}
    \bb{E}_x\left[\left(L_m^N(0)\right)^2\right] \;\leq\; \bb{E}_x\left[L_m^N(0)\right]\left(1+2\bb{E}_0\left[L_m^N(0)\right]\right).
\end{equation}
In particular, writing $U_N(m):=\sup_{x\in\bb{N}_0\cup\{\Delta\}}\bb{E}_x[L_m^N(0)]$, we have
\begin{equation}\label{eq:localTimeSecondMomentUniform}
    \sup_{x\in\bb{N}_0\cup\{\Delta\}}\bb{E}_x\left[\left(L_m^N(0)\right)^2\right] \;\leq\; U_N(m)\left(1+2U_N(m)\right).
\end{equation}
\end{lemma}

\begin{proof}
Expanding the square and using that indicators take values in $\{0,1\}$, we obtain

\begin{equation*}
    \bb{E}_x\left[\left(L_m^N(0)\right)^2\right] \;=\; \bb{E}_x\left[L_m^N(0)\right]+2\sum_{k=0}^{m-1}\bb{E}_x\left[\one_{[X_k^N=0]}\sum_{\ell=k+1}^{m-1}\one_{[X_\ell^N=0]}\right].
\end{equation*}
By the Markov property at time $k$,

\begin{equation*}
\begin{split}
    \bb{E}_x\left[\one_{[X_k^N=0]}\sum_{\ell=k+1}^{m-1}\one_{[X_\ell^N=0]}\right] \;&=\; \bb{E}_x\left[\one_{[X_k^N=0]}\right]\bb{E}_0\left[\sum_{r=1}^{m-1-k}\one_{[X_r^N=0]}\right]\\
    &\leq\; \bb{E}_x\left[\one_{[X_k^N=0]}\right]\bb{E}_0\left[L_m^N(0)\right],
\end{split}
\end{equation*}
where empty sums are understood to be zero. Summing over $k$ proves~\eqref{eq:localTimeSecondMoment}. Taking the supremum over the initial state gives~\eqref{eq:localTimeSecondMomentUniform}.
\end{proof}

\section{Tightness of the $(\alpha,\beta,A,B)$-boundary random walks}\label{sec:3}
This section is devoted to proving the tightness of the sequence of the rescaled random walks $\{B_t^N:t\in[0,T]\}_{N\in\bb{N}}$ in the Skorokhod space $\mc{D}([0,T],\bb{G})$ endowed with the $J_1$-topology.

To distinguish spatial motion from the passage to the cemetery state, define $g,h:\bb{G}\to\bb{R}$ by

\begin{equation*}
    g(x) \;:=\; x,\quad h(x) \;:=\; 0 \quad \text{for } x\in[0,\infty), \qquad g(\Delta) \;:=\; 0,\quad h(\Delta) \;:=\; 1.
\end{equation*}
We equip $\bb{G}$ with the metric

\begin{equation}\label{eq:stateSpaceMetric}
    d(x,y) \;:=\; |g(x)-g(y)|+|h(x)-h(y)|, \qquad x,y\in\bb{G}.
\end{equation}
Thus, $d(x,y)=|x-y|$ for $x,y\in[0,\infty)$, while $d(x,\Delta)=1+x$ for $x\in[0,\infty)$. This metric induces the usual topology on the half-line and makes $\Delta$ an isolated point. Define the spatial projection $Y_t^N:=g(B_t^N)$. We first establish tightness of $\{Y^N:N\in\bb{N}\}$ and then control the passage to the cemetery state.

For each $N\in\bb{N}$, let $\mc{G}_t^N:=\mc{F}_{\lfloor N^2t\rfloor}^{X^N}$, and denote by $\mc{T}_T^N$ the collection of all stopping times with respect to $(\mc{G}_t^N)_{t\geq0}$ taking values in $[0,T]$.

We will use Aldous's criterion~\cite[Theorem 1]{aldous1978} to establish tightness. That is, we will show that $B^N$ satisfies the following conditions:
\begin{enumerate}
    \item \label{en:tightness1} For any fixed $t\in[0,T]$, the sequence of random variables $\{B_t^N:N\in\bb{N}\}$ is tight.
    \item \label{en:tightness2} For every $\eps>0$,
    \begin{equation*}
        \lim_{\delta\downarrow0} \;\limsup_{N\to\infty} \;\sup_{\tau_N\in\mc{T}_T^N} \;\sup_{0\leq\theta\leq\delta} \Prob{}{d\big(B_{(\tau_N+\theta)\wedge T}^N,B_{\tau_N}^N\big)>\eps} \;=\; 0.
    \end{equation*}
\end{enumerate}

We first prove an auxiliary lemma that will be used to show the tightness of the sequence of processes $\{B^N:N\in\bb{N}\}$.

We begin with a basic result that will be used in all cases. Let $g : \{\Delta\}\cup\bb{N}_0 \to \bb{R}$ be a function and $X^{N}$ a simple random walk with boundary conditions at zero over $\{\Delta\}\cup\bb{N}_0$. 

Define
\begin{equation}\label{eq:discreteMartingale}
    \mc{M}_k^{N,g}\,:=\,\g{X_k^{N}}\,-\,\g{X_0^{N}}\,-\,\sum_{j=0}^{k-1}\hat{\msf{L}}_N\g{X_j^N}.
\end{equation}
\begin{lemma}\label{lem:discreteMartingale}
The process $\mc{M}_k^{N,g}$ defined in~\eqref{eq:discreteMartingale} is a Martingale with respect to the filtration $(\mc{G}_t^N :  N \geq 1)$.
\end{lemma}
\begin{proof}
Fix $N,k\in \bb{N}$. Thus
\begin{equation*}
\begin{split}
    \CExpec{\mc{M}_{k+1}^{N,g}\,-\,\mc{M}_{k}^{N,g}}{}{\mc{G}_t^N}\,&=\,\CExpec{\bigg(\g{X_{k+1}^{N}}\,-\,\g{X_0^{N}}\,-\,\sum_{j=0}^k\hat{\msf{L}}_N\g{X_j}\bigg)}{}{\mc{G}_t^N}\\
    &\quad-\,\CExpec{\bigg(\g{X_k^{N}}\,-\,\g{X_0^{N}}\,-\,\sum_{j=0}^{k-1}\hat{\msf{L}}_N\g{X_j^N}\bigg)}{}{\mc{G}_t^N}\\
    &=\,\CExpec{\g{X_{k+1}^N}\,-\,\g{X_k^N}}{}{ \mc{G}_t^N}\,-\,\CExpec{\hat{\msf{L}}_N\g{X_k^{N}}}{}{ \mc{G}_t^N}\\
    &=\,0.
\end{split}
\end{equation*}
Therefore, $\mc{M}_k^{N,g}$ is a martingale with respect to the filtration $(\mc{F}_n^{X^{N}} : n \geq 1)$.
\end{proof}

For each $N\in\bb{N}$, let $\mc{G}_t^N:=\mc{F}_{\lfloor N^2t\rfloor}^{X^N}$, and denote by $\mc{T}_T^N$ the collection of all stopping times with respect to $(\mc{G}_t^N)_{t\geq0}$ taking values in $[0,T]$.

\begin{lemma}\label{lem:tightSpatial}
Assume that
\begin{equation}\label{eq:condTIGHT}
    \sup_{N\in\bb{N}}\Expec{\left|Y_0^N\right|^2}{} \;<\; \infty.
\end{equation}
Then, for each $T>0$, the sequence $\{Y^N:N\in\bb{N}\}$ is tight in the Skorokhod space $\mc{D}([0,T],[0,\infty))$, endowed with the $J_1$-topology.
\end{lemma}

\begin{proof}
Let us first show that $Y_t^N$ is tight for each $t \in [0,T]$. Consider $f(x)=x^2$ for $x\in[0,\infty)$ and $f(\Delta)=0$. From Lemma~\ref{lem:discreteMartingale}, $\mc{M}_k^{N,f}$ is a martingale. By applying Markov's inequality, it follows that for any $M>0$,
\begin{equation}\label{eq:tight1.1}
\begin{split}
    \Prob{}{Y_t^N > M} \;&\leq\; \frac{1}{M^2}\Expec{(Y_t^N)^2}{}\\
    &=\; \frac{1}{M^2}\left[\Expec{(Y_0^N)^2}{} + \sum_{j=0}^{\lfloor N^2t\rfloor-1}\Expec{\hat{\msf{L}}_Nf\left(\pfrac{X_j^N}{N}\right)}{}\right].
\end{split}
\end{equation}
Here and below, we use the convention $\Delta/N=\Delta$. Since $f$ is continuous and $f(0)=0$, by computing it from the definition~\eqref{eq:discretegenerator}, we obtain the following expression for $\hat{\msf{L}}_Nf$
\begin{equation}\label{eq:tight1.2}
\begin{split}
    \hat{\msf{L}}_Nf\left(\pfrac{x}{N}\right) \;&=\; \frac{1}{2}\left(f \left(\pfrac{x+1}{N}\right)+f \left(\pfrac{x-1}{N}\right)-2f\left(\pfrac{x}{N}\right)\right)\one_{[x>0]}+\frac{B}{N^\beta}f\left(\pfrac{1}{N}\right)\one_{[x=0]}\\
    &=\; \frac{1}{N^2}\one_{[x>0]}+\frac{B}{N^{2+\beta}}\one_{[x=0]}\\
    &\leq\; \frac{1}{N^2}(1+B).
\end{split}
\end{equation}
The killing term vanishes because $f(\Delta)=f(0)=0$, and the same upper bound holds at $\Delta$, where the generator is zero.

Thus, by replacing the upper bound~\eqref{eq:tight1.2} into equation~\eqref{eq:tight1.1}, we obtain
\begin{equation*}
    \Prob{}{|Y_t^N}|>M \;\leq\; \frac{1}{M^2}\left[\Expec{\left(Y_0^N\right)^2}{}+N^2t\frac{1}{N^2}(1+B)\right].
\end{equation*}
From the assumption~\eqref{eq:condTIGHT}, we obtain that the right-hand side converges to zero as $M\to\infty$, uniformly in $N$, and hence, by choosing $M$ sufficiently large, the first criterion is satisfied.

For the second condition of Aldous' criterion, let $\tau_N \in \mc{T}_T^N$ and $T>0$. For $h \geq 0$, set $s = (\tau_N + h)\wedge T$. From an application of Markov's inequality, one can obtain that
\begin{equation}\label{eq:tight2.1}
    \Prob{}{|Y_s^N-Y_{\tau_N}^N}|\geq\eps \;\leq\; \frac{1}{\eps^2}\Expec{|Y_s^N-Y_{\tau_N}^N}|^2{}.
\end{equation}
Thus, from equation~\eqref{eq:discreteMartingale}, we have that
\begin{equation}\label{eq:tight2.2}
\begin{split}
    \left(Y_s^N \;-\; Y_{\tau_N}^N\right)^2 \;&=\; \Bigg(\mc{M}_{\lfloor sN^2\rfloor}^{N,g} \;-\; \mc{M}_{\lfloor \tau_N N^2\rfloor}^{N,g} \;+\; \sum_{j=\lfloor \tau_N N^2\rfloor}^{\lfloor sN^2\rfloor-1}\hat{\msf{L}}_N g\left(\pfrac{1}{N}X_j^N\right)\Bigg)^2\\
    &\leq\; 2\left(\mc{M}_{\lfloor sN^2\rfloor}^{N,g} \;-\; \mc{M}_{\lfloor \tau_N N^2\rfloor}^{N,g}\right)^2 \;+\; 2\Bigg(\sum_{j=\lfloor \tau_N N^2\rfloor}^{\lfloor sN^2\rfloor-1}\hat{\msf{L}}_N g\left(\pfrac{1}{N}X_j^N\right)\Bigg)^2.
\end{split}
\end{equation}

Recalling the definition of the discrete generator~\eqref{eq:discretegenerator}, applying it to $g(x)=x$ for $x\geq0$, with $g(\Delta)=0$, as we have done in~\eqref{eq:tight1.2}, it yields
\begin{equation}\label{eq:tight2.3}
    \hat{\msf{L}}_N \gn{g}{\pfrac{x}{N}} \;=\; \pfrac{B}{N^{1+\beta}}\one_{[x=0]}.
\end{equation}

Therefore, from~\eqref{eq:tight2.2} and \eqref{eq:tight2.3}, equation~\eqref{eq:tight2.1} becomes

\begin{equation}\label{eq:tight2.4}
\begin{split}
    &\Prob{}{|Y_s^N-Y_{\tau_N}^N|\geq\eps} \\
    &\overset{\eqref{eq:tight2.2}}{\leq}\frac{2}{\eps^2}\Expec{\left(\mc{M}_{\lfloor sN^2\rfloor}^{N,g} \;-\; \mc{M}_{\lfloor \tau_N N^2\rfloor}^{N,g}\right)^2}{} \;+\; \frac{2}{\eps^2}\bb{E}\bigg[\bigg(\sum_{j=\lfloor \tau_N N^2\rfloor}^{\lfloor sN^2\rfloor-1}\hat{\msf{L}}_N g\big(\pfrac{1}{N}X_j^N\big)\bigg)^2\bigg]\\
    &\overset{\eqref{eq:tight2.3}}{=}\frac{2}{\eps^2}\Expec{\left(\mc{M}_{\lfloor sN^2\rfloor}^{N,g} \;-\; \mc{M}_{\lfloor \tau_N N^2\rfloor}^{N,g}\right)^2}{} \;+\; \frac{2B^2}{\eps^2N^{2+2\beta}}\bb{E}\bigg[\bigg(\sum_{j=\lfloor \tau_N N^2\rfloor}^{\lfloor sN^2\rfloor-1}\one_{[X_j^N=0]}\bigg)^2\bigg].
\end{split}
\end{equation}

We begin by estimating the first term on the right-hand side of Equation~\eqref{eq:tight2.4}, namely the expectation involving the martingale difference. Set $K:=\lfloor \tau_N N^2\rfloor$ and $J:=\lfloor sN^2\rfloor$. Since $K$ and $J$ are bounded stopping times for the discrete filtration, the indicators $\one_{[K\leq j<J]}$ are $\mc{F}_j^X$-measurable.

Now, writing $\xi_{j+1}:=\mc{M}_{j+1}^{N,g}-\mc{M}_j^{N,g}$ and $H:=\lfloor N^2T\rfloor$, we have

\begin{equation*}
    \mc{M}_J^{N,g}-\mc{M}_K^{N,g} \;=\; \sum_{j=0}^{H-1}\one_{[K\leq j<J]}\xi_{j+1}.
\end{equation*}
For $i<j$, the random variable $\one_{[K\leq i<J]}\xi_{i+1}\one_{[K\leq j<J]}$ is $\mc{F}_j^X$-measurable. Since $\bb{E}[\xi_{j+1}\mid\mc{F}_j^X]=0$, all cross terms have zero expectation. Consequently,

\begin{equation*}
\begin{split}
    \bb{E}\Big[\big(\mc{M}_J^{N,g}-\mc{M}_K^{N,g}\big)^2\Big] \;&=\; \sum_{j=0}^{H-1}\bb{E}\Big[\one_{[K\leq j<J]}\bb{E}\big[\xi_{j+1}^2\mid\mc{F}_j^X\big]\Big]\\
    &=\; \bb{E}\bigg[\sum_{j=K}^{J-1}\bb{E}\big[\xi_{j+1}^2\mid\mc{F}_j^X\big]\bigg]\,.
\end{split}
\end{equation*}
Hence, 
\begin{equation}\label{eq:tight2.5}
\begin{split}
    &\bb{E}\Big[\big(\mc{M}_J^{N,g}-\mc{M}_K^{N,g}\big)^2\Big] \;=\; \bb{E}\bigg[\sum_{j=K}^{J-1}\CExpec{\Big(g\big(\pfrac{1}{N}X_{j+1}^N\big) \;-\; g\big(\pfrac{1}{N}X_j^N\big) \;-\; \hat{\msf{L}}_N g\big(\pfrac{X_j^N}{N}\big)\Big)^2}{}{\mc{F}_j^X}\bigg]\\
    &\quad\leq\; 2\bb{E}\bigg[\sum_{j=K}^{J-1}\CExpec{\Big(g\big(\pfrac{1}{N}X_{j+1}^N\big) \;-\; g\big(\pfrac{1}{N}X_j^N\big)\Big)^2}{}{\mc{F}_j^X}\bigg] \;+\; 2\bb{E}\bigg[\sum_{j=K}^{J-1}\Big(\hat{\msf{L}}_N g\big(\pfrac{1}{N}X_j^N\big)\big)^2\bigg]\\
    &\quad\overset{\eqref{eq:tight2.3}}{\leq}\; \frac{2}{N^2}\bb{E}\bigg[\sum_{j=K}^{J-1}\one_{[X_j^N\in\bb{N}]}\bigg] \;+\; \frac{2B}{N^{2+\beta}}\bb{E}\bigg[\sum_{j=K}^{J-1}\one_{[X_j^N=0]}\bigg] \;+\; \frac{2B^2}{N^{2+2\beta}}\bb{E}\bigg[\sum_{j=K}^{J-1}\one_{[X_j^N=0]}\bigg]\\
    &\quad\leq\; 2\Big(h \;+\; \pfrac{1}{N^2}\Big)\Big(1 \;+\; \pfrac{B}{N^\beta} \;+\; \pfrac{B^2}{N^{2\beta}}\Big).
\end{split}
\end{equation}

The preceding computation uses that $g(0)=g(\Delta)=0$ and that $\Delta$ is absorbing. Set $m:=\lfloor N^2h\rfloor+1$. Since $J-K\leq m$, the strong Markov property at $K$, Lemma~\ref{lem:localTimeSecondMoment} and the uniform first-moment estimate give

\begin{equation*}
    \bb{E}\bigg[\Big(\sum_{j=K}^{J-1}\one_{[X_j^N=0]}\Big)^2\bigg] \;\leq\; \sup_x\bb{E}_x\big[(L_m^N(0))^2\big] \;\leq\; C\big(1+N^{2\beta}m\big),
\end{equation*}
where $C$ is independent of $N$, $h$ and $\tau$. Finally, from equations~\eqref{eq:tight2.4} and \eqref{eq:tight2.5} we bound~\eqref{eq:tight2.1} as follows

\begin{equation*}
\begin{split}
    \Prob{}{|Y_s^N-Y_\tau^N|\geq\eps} \;&\leq\; \frac{4}{\eps^2}\Big(h \;+\; \pfrac{1}{N^2}\Big)\Big(1 \;+\; \pfrac{B}{N^\beta} \;+\; \pfrac{B^2}{N^{2\beta}}\Big)\\
    &\quad+\; \frac{2B^2C}{\eps^2}\Big(\pfrac{1}{N^{2+2\beta}} \;+\; h \;+\; \pfrac{1}{N^2}\Big).
\end{split}
\end{equation*}
Since $\beta\geq0$, making $N\to\infty$ and then $h\to0$, uniformly over $\tau_N\in\mc{T}_T^N$ and increments bounded by $h$, we obtain

\begin{equation*}
    \lim_{\delta\downarrow0} \;\limsup_{N\to\infty} \;\sup_{\tau_N\in\mc{T}_T^N} \;\sup_{0\leq h\leq\delta} \Prob{}{\big|Y_{(\tau_N+h)\wedge T}^N-Y_{\tau_N}^N\big|\geq\eps} \;=\; 0.
\end{equation*}
Thus, the second condition is satisfied. Applying Aldous' criterion~\cite{aldous1978}, it follows that $\{Y^N:N\in\bb{N}\}$ is tight in the Skorokhod space $\mc{D}([0,T],[0,\infty))$, endowed with the $J_1$-topology.
\end{proof}

We next control the probability of reaching the cemetery state during a short interval after a stopping time.

\begin{lemma}\label{lem:cemeteryControl}
Assume that $A,B>0$, $\alpha,\beta\geq0$, and $\alpha\geq\min\{2,\beta+1\}$. Then, for every $T>0$ and any sequence of initial distributions,
\begin{equation}\label{eq:cemeteryControl}
    \lim_{\delta\downarrow0} \;\limsup_{N\to\infty} \;\sup_{\tau_N\in\mc{T}_T^N} \;\sup_{0\leq\theta\leq\delta} \bb{P}\left(B_{\tau_N}^N\neq\Delta,\; B_{(\tau_N+\theta)\wedge T}^N=\Delta\right) \;=\; 0.
\end{equation}
\end{lemma}

\begin{proof}
Fix $\tau_N\in\mc{T}_T^N$ and $\theta\geq0$, and set $s:=(\tau_N+\theta)\wedge T$, $K:=\lfloor N^2\tau_N\rfloor$, $J:=\lfloor N^2s\rfloor$ and $m:=\lfloor N^2\theta\rfloor+1$. Then $K$ and $J$ are bounded stopping times for the discrete filtration, and $J-K\leq m$. Since $\Delta$ is absorbing and can only be reached from zero,
\begin{equation*}
    \one_{[X_K^N\neq\Delta,\;X_J^N=\Delta]} \;=\; \sum_{j=K}^{J-1}\one_{[X_j^N=0,\;X_{j+1}^N=\Delta]}.
\end{equation*}
Using that $\one_{[K\leq j<J]}$ is $\mc{F}_j^X$-measurable and that the killing probability at zero is $A/N^\alpha$, conditional expectation gives
\begin{equation}\label{eq:cemeteryLocalTime}
    \bb{P}\big(B_{\tau_N}^N\neq\Delta,\;B_s^N=\Delta\big) \;=\; \frac{A}{N^\alpha}\bb{E}\bigg[\sum_{j=K}^{J-1}\one_{[X_j^N=0]}\bigg].
\end{equation}

If $\beta\geq1$, the assumption implies $\alpha\geq2$. Bounding the sum by the number of steps, we obtain
\begin{equation*}
    \bb{P}\big(B_{\tau_N}^N\neq\Delta,\;B_s^N=\Delta\big) \;\leq\; \frac{Am}{N^\alpha} \;\leq\; AN^{2-\alpha}\big(\theta+N^{-2}\big) \;\leq\; A\big(\theta+N^{-2}\big).
\end{equation*}

If $0\leq\beta<1$, the assumption implies $\alpha\geq\beta+1$. By the strong Markov property at $K$ and the uniform first-moment estimate in Corollary~\ref{cor:localTimeBound}, we have
\begin{equation*}
\begin{split}
    \bb{E}\bigg[\sum_{j=K}^{J-1}\one_{[X_j^N=0]}\bigg] \;&\leq\; \bb{E}\bigg[\bb{E}_{X_K^N}\Big[\sum_{r=0}^{m-1}\one_{[X_r^N=0]}\Big]\bigg]\\
    &\leq\; 2+cN^\beta\sqrt m,
\end{split}
\end{equation*}
where $c:=2D_*/(B\sqrt{2\pi})$ and $m = \lfloor N^2\theta\rfloor+1$. Substituting into~\eqref{eq:cemeteryLocalTime}, we obtain
\begin{equation*}
\begin{split}
    \bb{P}\big(B_{\tau_N}^N\neq\Delta,\;B_s^N=\Delta\big) \;&\leq\; \frac{2A}{N^\alpha}+AcN^{\beta+1-\alpha}\sqrt{\theta+N^{-2}}\\
    &\leq\; \frac{2A}{N^\alpha}+Ac\sqrt{\theta+N^{-2}}.
\end{split}
\end{equation*}
Both bounds are independent of $\tau_N$. Taking the supremum over $\tau_N\in\mc{T}_T^N$ and $0\leq\theta\leq\delta$, then letting $N\to\infty$ and $\delta\downarrow0$, proves~\eqref{eq:cemeteryControl}.
\end{proof}

Combining the spatial estimates with Lemma~\ref{lem:cemeteryControl}, we obtain tightness on the full state space.

\begin{lemma}\label{lem:tight}
Assume that $A,B>0$, $\alpha,\beta\geq0$, and $\alpha\geq\min\{2,\beta+1\}$. Suppose also that condition~\eqref{eq:condTIGHT} holds. Then, for every $T>0$, the sequence $\{B^N:N\in\bb{N}\}$ is tight in the Skorokhod space $\mc{D}([0,T],\bb{G})$, endowed with the $J_1$-topology associated with the metric~\eqref{eq:stateSpaceMetric}.
\end{lemma}

\begin{proof}
For each $M>0$, the set $C_M:=[0,M]\cup\{\Delta\}$ is compact in $\bb{G}$. Since $Y_t^N=g(B_t^N)$,
\begin{equation*}
    \bb{P}\big(B_t^N\notin C_M\big) \;=\; \bb{P}\big(Y_t^N>M\big).
\end{equation*}
The fixed-time estimate in the proof of Lemma~\ref{lem:tightSpatial} therefore establishes criterion~\ref{en:tightness1}.

For criterion~\ref{en:tightness2}, fix $\tau_N\in\mc{T}_T^N$ and $\theta\geq0$, and set $s:=(\tau_N+\theta)\wedge T$. Since $\Delta$ is absorbing, the definition of $d$ gives
\begin{equation*}
    d\big(B_s^N,B_{\tau_N}^N\big) \;=\; \big|Y_s^N-Y_{\tau_N}^N\big|+\one_{[B_{\tau_N}^N\neq\Delta,\;B_s^N=\Delta]}.
\end{equation*}
Consequently, for every $\eps>0$,
\begin{equation*}
\begin{split}
    \bb{P}\big(d(B_s^N,B_{\tau_N}^N)>\eps\big) \;&\leq\; \bb{P}\big(|Y_s^N-Y_{\tau_N}^N|>\eps\big)\\
    &\quad+\; \bb{P}\big(B_{\tau_N}^N\neq\Delta,\;B_s^N=\Delta\big).
\end{split}
\end{equation*}
Taking the supremum over $\tau_N\in\mc{T}_T^N$ and $0\leq\theta\leq\delta$, the increment estimate proved in Lemma~\ref{lem:tightSpatial} and Lemma~\ref{lem:cemeteryControl} yield
\begin{equation*}
    \lim_{\delta\downarrow0} \;\limsup_{N\to\infty} \;\sup_{\tau_N\in\mc{T}_T^N} \;\sup_{0\leq\theta\leq\delta} \bb{P}\big(d(B_{(\tau_N+\theta)\wedge T}^N,B_{\tau_N}^N)>\eps\big) \;=\; 0.
\end{equation*}
Thus, criterion~\ref{en:tightness2} holds, and Aldous' criterion~\cite{aldous1978} concludes the proof.
\end{proof}

\section{Characterization of the limit points}\label{sec:4}

This section is devoted to characterizing the limit points of the family $\{B^N:N\geq1\}$ of rescaled random walks using martingale problem techniques. It is divided into seven subsections, each corresponding to a different Brownian-type limiting process. Having established tightness in Section~\ref{sec:3}, we will show that each limit point solves the corresponding martingale problem. We first present the argument in detail for the mixed case and then address the particularities of the remaining cases.

Let $B$ denote a process whose law is a subsequential limit of the laws of $B^N$, as provided by Lemma~\ref{lem:tight}. Throughout this section, we work along this convergent subsequence, which we do not relabel. Let $\mc{L}$ denote the candidate limiting generator. For each $f \in \mf{D}(\mc{L})$, define
\begin{equation}\label{eq:martingale}
    \mc{M}_t^f \;:=\; \f{B_t} \;-\; \f{B_0} \;-\; \int_0^t\mc{L}\f{B_s}\dd s\,.
\end{equation}

In what follows, we will specify the operator $\mc{L}$ and its domain for each case of Brownian-type processes on the half-line. To avoid overloading notation, we will reuse the symbols $\mc{M}_t^f$, $B$ and $\mc{L}$ in each of the following subsections. For any $t \in [0,T]$, we also define $k_t := \lfloor N^2t\rfloor$.

Given $f \in \mc{C}_0^2(\bb{G})$, extended by $f(\Delta)=0$, we denote by $f_N := \pi_Nf$ its pullback to the unscaled lattice, where
\begin{equation*}
    \pi_N \;:\; \mc{C}_0(\bb{G}) \;\to\; \mc{C}_0(\tilde{\bb{G}})
\end{equation*}
is defined by
\begin{equation*}
    \gn{f}{x}{N} \;=\;
    \begin{cases}
        \f{\pfrac{x}{N}}, &\text{ if } x \in \tilde{\bb{G}}\setminus\{\Delta\}\,,\\
        0, &\text{ if } x = \Delta\,.
    \end{cases}
\end{equation*}
The rescaled operator $\msf{L}_N$ and the discrete generator $\hat{\msf{L}}_N$ are related by

\begin{equation}\label{eqdef:generatorXB}
    \msf{L}_N\f{\pfrac{x}{N}} \;=\; N^2\hat{\msf{L}}_N\gn{f}{x}{N}\,, \qquad x \in \tilde{\bb{G}}\setminus\{\Delta\}\,,
\end{equation}

and both operators vanish at $\Delta$. Here and below, $\msf{L}_Nf$ means that $\msf{L}_N$ acts on the restriction of $f$ to $\tilde{\bb{G}}_N$. Therefore, for any $f\in \mc{C}_0^2(\bb{G})$, and recalling the definition of $\msf{L}_N$ from~\eqref{eq:boundaryRWgen}, it follows from a Taylor expansion that
\begin{equation}\label{eq:Brownian.generator}
\begin{split}
    \msf{L}_N\f{\pfrac{x}{N}} \;&=\; \frac{1}{2}\Delta^N\f{\pfrac{x}{N}}\one_{[x > 0]} \;+\; \left[\frac{B}{N^{\beta - 2}}\left[\f{\pfrac{1}{N}} \;-\; \f{0}\right] \;-\; \frac{A}{N^{\alpha - 2}}\f{0}\right]\one_{[x=0]}\\
    &=\; \left[\frac{1}{2}\gn{f''}{\pfrac{x}{N}}{} \;+\; \of{1}\right]\one_{[x>0]} \;+\; \left[\frac{B}{N^{\beta - 1}}\left[f'(0) \;+\; \of{1}\right] \;-\; \frac{A}{N^{\alpha - 2}}\f{0}\right]\one_{[x=0]}\,,
\end{split}
\end{equation}
where $x \in \tilde{\bb{G}}\setminus\{\Delta\}$, and the remainder in the interior is understood pointwise. Uniform estimates for functions in the domain of the limiting generator will be established below.

\subsection{Generator estimates and convergence of the compensators}
\label{subsec:compensatorConvergence}

In this subsection, we analyze separately each parameter regime appearing in Theorem~\ref{thm:main}. For each regime, we specify the generator $\mc{L}$ of the candidate limiting process and its domain $\mf{D}(\mc{L})$, and use the corresponding boundary condition to estimate the difference between the rescaled discrete generator $\msf{L}_N$ and $\mc{L}$. Combined, when necessary, with the local-time estimate from Corollary~\ref{cor:localTimeBound}, these bounds will allow us to prove the convergence of the discrete compensators to their limiting counterparts.

Throughout this subsection, we work along the subsequence fixed at the beginning of this section. We denote by $\bb{P}^N$ and $\bb{E}^N$ probability and expectation under the law of $B^N$, and by $\bb{P}$ and $\bb{E}$ the corresponding objects under the law of $B$.

\subsubsection{Mixed Brownian motion: the $(2,1,A,B)$ regime} The mixed Brownian motion is a Feller process on the half-line, with continuous paths before its lifetime, which is characterized by the generator $\mc{L} = \frac{1}{2}\frac{d^2}{dx^2}$ and the domain $\mf{D}(\mc{L})$ as follows:
\begin{equation*}
    \mf{D}(\mc{L}) \;=\; \bigg\{f\in\mc{C}_0^2(\bb{G}):f''\in\mc{C}_0([0,\infty)),\;\text{ and }\;Af(0) \;-\; Bf'(0) \;+\; \frac{1}{2}f''(0) \;=\; 0\bigg\}\,.
\end{equation*}
We extend $f$ and $\mc{L}f$ to the cemetery state by setting $f(\Delta)=\mc{L}f(\Delta)=0$.

In order to compare the discrete dynamics with the limiting continuous generator, fix an arbitrary test function $f \in \mf{D}(\mc{L})$. For $x \in \tilde{\bb{G}}_N\setminus\{\Delta\}$, recalling that the generator of the boundary random walk acts as in~\eqref{eq:Brownian.generator}, with parameters $\alpha = 2$ and $\beta = 1$, we obtain
\begin{equation}\label{eq:generatorMBM}
\begin{split}
    \msf{L}_N\gn{f}{x}{} \;&=\; \bigg(\pfrac{1}{2}f''(x) \;+\; \of{1}\bigg)\one_{[x>0]} \;+\; \big(Bf'(0) \;-\; A\f{0} \;+\; \of{1}\big)\one_{[x=0]}\\
    &=\; \big(\mc{L}\gn{f}{x}{} \;+\; \of{1}\big)\one_{[x>0]} \;+\; \big(\mc{L}\f{0} \;+\; \of{1}\big)\one_{[x=0]}\,.
\end{split}
\end{equation}
The next lemma shows that the remainders in~\eqref{eq:generatorMBM} vanish uniformly over the rescaled lattice.

\begin{lemma}\label{lem:uniformGeneratorMBM}
Let $f\in\mf{D}(\mc{L})$ and define $h_N(y):=\msf{L}_Nf(y) \;-\; \mc{L}f(y)$ for $y\in\tilde{\bb{G}}_N$, with $h_N(\Delta)=0$. Then
\begin{equation*}
    \sup_{y\in\tilde{\bb{G}}_N}|h_N(y)| \;\longrightarrow\; 0\,.
\end{equation*}
In particular, there exists a constant $C_f<\infty$, independent of $N$, such that
\begin{equation}\label{eq:uniformGeneratorBoundMBM}
    \sup_{y\in\tilde{\bb{G}}_N}|\msf{L}_Nf(y)| \;\leq\; C_f
\end{equation}
for every $N$ for which the random walk is defined.
\end{lemma}
\begin{proof}
Since $f''\in\mc{C}_0([0,\infty))$, it is bounded and uniformly continuous. Define its modulus of continuity by
\begin{equation*}
    \omega_{f''}(r) \;:=\; \sup_{\substack{x,y\geq0\\|x-y|\leq r}}|f''(x)-f''(y)|\,.
\end{equation*}
Then $\omega_{f''}(r)\to0$ as $r\downarrow0$. For $y\in\tilde{\bb{G}}_N\setminus\{0,\Delta\}$, Taylor's formula with integral remainder gives
\begin{equation*}
\begin{split}
    f(y+1/N) \;-\; f(y) \;&=\; \frac{1}{N}f'(y) \;+\; \int_0^{1/N}(1/N-u)f''(y+u)\dd u\,,\\
    f(y-1/N) \;-\; f(y) \;&=\; -\frac{1}{N}f'(y) \;+\; \int_0^{1/N}(1/N-u)f''(y-u)\dd u\,.
\end{split}
\end{equation*}
Adding these identities and changing variables, we obtain
\begin{equation*}
    \msf{L}_Nf(y) \;=\; \frac{1}{2}\int_0^1(1-v)\big[f''(y+v/N) \;+\; f''(y-v/N)\big]\dd v\,.
\end{equation*}
Consequently,
\begin{equation*}
\begin{split}
    |h_N(y)| \;&\leq\; \frac{1}{2}\int_0^1(1-v)\big[|f''(y+v/N)-f''(y)| \;+\; |f''(y-v/N)-f''(y)|\big]\dd v\\
    &\leq\; \frac{1}{2}\omega_{f''}(1/N)\,.
\end{split}
\end{equation*}

At the origin, Taylor's formula yields
\begin{equation*}
\begin{split}
    \msf{L}_Nf(0) \;&=\; BN\big[f(1/N)-f(0)\big] \;-\; Af(0)\\
    &=\; Bf'(0) \;-\; Af(0) \;+\; BN\int_0^{1/N}(1/N-u)f''(u)\dd u\,.
\end{split}
\end{equation*}
Using the boundary condition $Bf'(0)-Af(0)=\mc{L}f(0)$, we conclude that
\begin{equation*}
    |h_N(0)| \;\leq\; BN\|f''\|_\infty\int_0^{1/N}(1/N-u)\dd u \;=\; \frac{B}{2N}\|f''\|_\infty\,.
\end{equation*}
Since $h_N(\Delta)=0$, it follows that
\begin{equation*}
    \sup_{y\in\tilde{\bb{G}}_N}|h_N(y)| \;\leq\; \max\bigg\{\frac{1}{2}\omega_{f''}(1/N),\frac{B}{2N}\|f''\|_\infty\bigg\} \;\longrightarrow\; 0\,.
\end{equation*}
Finally, the preceding bound is uniform in $N$ and $\mc{L}f$ is bounded, which proves~\eqref{eq:uniformGeneratorBoundMBM}.
\end{proof}

The uniform generator estimate obtained above also controls the
difference between the discrete and limiting compensators. We now
use this estimate to establish the convergence of the associated
martingales.

\begin{lemma}\label{lem:martingaleMBM}
Let $f \in \mf{D}(\mc{L})$. Then $\mc{M}_{k_t}^{N,f} \Rightarrow \mc{M}_t^f$ as $N \to \infty$ in the Skorokhod space $\mc{D}([0,T],\bb{R})$. Moreover,
\begin{equation*}
    \big(B^N,\mc{M}_{k_\cdot}^{N,f}\big) \;\Rightarrow\; \big(B,\mc{M}_\cdot^f\big)
\end{equation*}
in $\mc{D}([0,T],\bb{G})\times\mc{D}([0,T],\bb{R})$, endowed with the product of the $J_1$ topologies.
\end{lemma}
\begin{proof}
Let us first rewrite the discrete martingale. Fix $f\in\mf{D}(\mc{L})$. Then
\begin{equation*}
    \mc{M}_{k_t}^{N,f} \;=\; f_N(X_{k_t}^N) \;-\; f_N(X_0^N) \;-\; \sum_{j=0}^{k_t-1}\hat{\msf{L}}_N\gn{f}{X_j^N}{N}\,.
\end{equation*}
Define
\begin{equation*}
    A_{k_t}^{\hat{\msf{L}}_N}(X_\cdot^N) \;:=\; \sum_{j=0}^{k_t-1}\hat{\msf{L}}_Nf_N(X_j^N)\,.
\end{equation*}
Then, the martingale $\mc{M}_{k_t}^{N,f}$ can be rewritten as
\begin{equation*}
    \mc{M}_{k_t}^{N,f} \;=\; f(B_t^N) \;-\; f(B_0^N) \;-\; A_{k_t}^{\hat{\msf{L}}_N}(X_\cdot^N)\,.
\end{equation*}
From~\eqref{eqdef:generatorXB}, we have
\begin{equation*}
    A_{k_t}^{\hat{\msf{L}}_N}(X_\cdot^N) \;=\; \frac{1}{N^2}\sum_{j=0}^{k_t-1}\msf{L}_Nf(B_{j/N^2}^N)\,.
\end{equation*}
Writing $s_j=j/N^2$ and $\Delta s_j=1/N^2$, we obtain
\begin{equation}\label{eq:compensador}
    A_{k_t}^{\hat{\msf{L}}_N}(X_\cdot^N) \;=\; \sum_{j=0}^{k_t-1}\mc{L}f(B_{s_j}^N)\Delta s_j \;+\; \sum_{j=0}^{k_t-1}h_N(B_{s_j}^N)\frac{1}{N^2}\,,
\end{equation}
where $h_N(y):=\msf{L}_Nf(y) \;-\; \mc{L}f(y)$ for $y\in\tilde{\bb{G}}_N\setminus\{\Delta\}$, and $h_N(\Delta) = 0$. The second term in~\eqref{eq:compensador} will be referred to as the \textit{error term}.

We first analyze this error uniformly over $t\in[0,T]$. Since $f\in\mf{D}(\mc{L})$, Lemma~\ref{lem:uniformGeneratorMBM} gives
\begin{equation}\label{eq:generatorMBM2}
    a_N \;:=\; \sup_{y\in\tilde{\bb{G}}_N\setminus\{0,\Delta\}}|h_N(y)| \;\longrightarrow\; 0
    \qquad\text{and}\qquad
    b_N \;:=\; |h_N(0)| \;\longrightarrow\; 0\,.
\end{equation}
For every $t\in[0,T]$, the error term can be rewritten as
\begin{equation}\label{eq:compensator2}
\begin{split}
    \sum_{j=0}^{k_t-1}h_N(B_{s_j}^N)\frac{1}{N^2}
    \;&=\; \sum_{j=0}^{k_t-1}h_N(B_{s_j}^N)\frac{1}{N^2}\one_{[B_{s_j}^N\notin\{0,\Delta\}]} \;+\; \sum_{j=0}^{k_t-1}h_N(0)\frac{1}{N^2}\one_{[B_{s_j}^N=0]}\\
    &=\; \sum_{j=0}^{k_t-1}\sum_{y\in\tilde{\bb{G}}_N\setminus\{0,\Delta\}}h_N(y)\frac{1}{N^2}\one_{[B_{s_j}^N=y]} \;+\; \sum_{j=0}^{k_t-1}h_N(0)\frac{1}{N^2}\one_{[B_{s_j}^N=0]}\,.
\end{split}
\end{equation}
Moreover,
\begin{equation*}
    \sup_{t\in[0,T]}\bigg|\sum_{j=0}^{k_t-1}h_N(B_{s_j}^N)\frac{1}{N^2}\bigg| \;\leq\; \sum_{j=0}^{k_T-1}|h_N(B_{s_j}^N)|\frac{1}{N^2}\,.
\end{equation*}
Then, applying Markov's inequality we obtain
\begin{equation}\label{eq:compensator3}
\begin{split}
    &\bb{P}^N\bigg(\sup_{t\in[0,T]}\bigg|\sum_{j=0}^{k_t-1}h_N(B_{s_j}^N)\frac{1}{N^2}\bigg|>\eps\bigg) \;\leq\; \frac{1}{\eps}\bb{E}^N\bigg[\sup_{t\in[0,T]}\bigg|\sum_{j=0}^{k_t-1}h_N(B_{s_j}^N)\frac{1}{N^2}\bigg|\bigg]\\
    &\qquad\leq\; \frac{1}{\eps}\bigg[\sum_{j=0}^{k_T-1}\bb{E}^N\big[|h_N(B_{s_j}^N)|\pfrac{1}{N^2}\one_{[B_{s_j}^N\notin\{0,\Delta\}]}\big] \;+\; \sum_{j=0}^{k_T-1}\bb{E}^N\big[|h_N(0)|\pfrac{1}{N^2}\one_{[B_{s_j}^N=0]}\big]\bigg]\\
    &\qquad=\; \frac{1}{\eps}\bigg[\sum_{j=0}^{k_T-1}\sum_{y\in\tilde{\bb{G}}_N\setminus\{0,\Delta\}}|h_N(y)|\pfrac{1}{N^2}\bb{P}^N(B_{s_j}^N=y) \;+\; \sum_{j=0}^{k_T-1}|h_N(0)|\pfrac{1}{N^2}\bb{P}^N(B_{s_j}^N=0)\bigg]\\
    &\qquad\overset{\eqref{eq:generatorMBM2}}{\leq}\; \frac{a_N}{\eps}\sum_{j=0}^{k_T-1}\pfrac{1}{N^2} \;+\; \frac{b_N}{\eps}\sum_{j=0}^{k_T-1}\pfrac{1}{N^2}\bb{P}^N(B_{s_j}^N=0)\\
    &\qquad\leq\; \frac{1}{\eps}\bigg[a_N\frac{k_T}{N^2} \;+\; b_N\frac{k_T}{N^2}\bigg]\\
    &\qquad\leq\; \frac{T}{\eps}(a_N+b_N) \;\longrightarrow\; 0\,.
\end{split}
\end{equation}
The interchange of expectation and summation is justified by Tonelli's theorem, since all terms are nonnegative. This shows that
\begin{equation*}
    \sup_{t\in[0,T]}\bigg|\sum_{j=0}^{k_t-1}h_N(B_{s_j}^N)\frac{1}{N^2}\bigg| \;\longrightarrow\; 0
\end{equation*}
in probability.

We now consider the first sum in~\eqref{eq:compensador}. Since $B_s^N=B_{s_j}^N$ for $s\in[s_j,s_{j+1})$,
\begin{equation*}
    \sum_{j=0}^{k_t-1}\mc{L}f(B_{s_j}^N)\Delta s_j \;=\; \int_0^{k_t/N^2}\mc{L}f(B_s^N)\dd s\,.
\end{equation*}
Since $0\leq t-k_t/N^2<1/N^2$, we have
\begin{equation*}
    \sup_{t\in[0,T]}\bigg|\sum_{j=0}^{k_t-1}\mc{L}f(B_{s_j}^N)\Delta s_j \;-\; \int_0^t\mc{L}f(B_s^N)\dd s\bigg| \;\leq\; \frac{\|\mc{L}f\|_\infty}{N^2}\,.
\end{equation*}
Combining this estimate with~\eqref{eq:compensator3}, we conclude that
\begin{equation}\label{eq:uniformMartingaleErrorMBM}
    \sup_{t\in[0,T]}\bigg|\mc{M}_{k_t}^{N,f} \;-\; \bigg[f(B_t^N) \;-\; f(B_0^N) \;-\; \int_0^t\mc{L}f(B_s^N)\dd s\bigg]\bigg| \;\longrightarrow\; 0
\end{equation}
in probability.

To finish the proof, define
\begin{equation*}
    \Phi_f(z)(t) \;:=\; f(z(t)) \;-\; f(z(0)) \;-\; \int_0^t\mc{L}f(z(s))\dd s\,.
\end{equation*}
We verify that $\Phi_f:\mc{D}([0,T],\bb{G})\to\mc{D}([0,T],\bb{R})$ is continuous for the $J_1$ topology. Suppose that $z_n\to z$ in $J_1$. Continuity of $f$ implies $f\circ z_n\to f\circ z$ in $J_1$: the time changes for $z_n\to z$ also give convergence after composition with $f$, using the compact closure of the range of $z$. Moreover, $z_n(s)\to z(s)$ at every continuity time of $z$, hence for Lebesgue-almost every $s$. Since $\mc{L}f$ is bounded and continuous, dominated convergence gives
\begin{equation*}
\begin{split}
    \sup_{t\in[0,T]}\bigg|\int_0^t\mc{L}f(z_n(s))\dd s \;-\; \int_0^t\mc{L}f(z(s))\dd s\bigg|
    &\leq\; \int_0^T|\mc{L}f(z_n(s))-\mc{L}f(z(s))|\dd s\\
    &\longrightarrow\; 0\,.
\end{split}
\end{equation*}
The limiting integral is continuous in time, and $f(z_n(0))\to f(z(0))$. Subtracting these terms therefore preserves the $J_1$ convergence, proving continuity of $\Phi_f$.

By the mapping theorem,
\begin{equation*}
    \big(B^N,\Phi_f(B^N)\big) \;\Rightarrow\; \big(B,\Phi_f(B)\big) \;=\; \big(B,\mc{M}_\cdot^f\big)\,.
\end{equation*}
Together with~\eqref{eq:uniformMartingaleErrorMBM}, this proves the asserted joint convergence and, in particular, the convergence of $\mc{M}_{k_\cdot}^{N,f}$ to $\mc{M}_\cdot^f$ in $\mc{D}([0,T],\bb{R})$.
\end{proof}

\normalcolor
\subsubsection{Sticky Brownian motion: the $(\alpha,1,A,B)$ regime with $\alpha>2$}

The sticky Brownian motion is a Feller diffusion process on the half-line, characterized by the generator $\mc{L} = \pfrac{1}{2}\pfrac{d^2}{dx^2}$ and the domain $\mf{D}(\mc{L})$ as follows:
\begin{equation*}
    \mf{D}(\mc{L})\,=\,\bigg\{f\in\mc{C}_0^2(\bb{G}):f''\in\mc{C}_0([0,\infty)),\,\text{ and }\,-Bf'(0)\,+\,\frac{1}{2}f''(0)\,=\,0\bigg\}.
\end{equation*}

In order to compare the discrete dynamics with the limiting continuous generator, fix an arbitrary test function $f \in \mf{D}(\mc{L})$. Recalling the generator of boundary random walk~\eqref{eq:Brownian.generator}, we obtain, for $x \in \tilde{\bb{G}}_N\setminus\{\Delta\}$ that
\begin{equation}\label{eq:generatorSBM}
\begin{split}
    \msf{L}_N\gn{f}{x}{} \;&=\; \bigg(\pfrac{1}{2}f''(x) \;+\;r_N(x)\bigg)\one_{[x>0]} \;+\;
    \bigg(Bf'(0)\,-\,\pfrac{A}{N^{\alpha-2}}f(0) \;+\;r_N(0)\bigg)\one_{[x=0]}\\
    &=\; \big(\mc{L}\gn{f}{x}{} \;+\;r_N(x)\big)\one_{[x>0]} \;+\; \bigg(\mc{L}\f{0}\,-\,\pfrac{A}{N^{\alpha-2}}f(0) \;+\; r_N(0)\bigg)\one_{[x=0]}\,.
\end{split}
\end{equation}
Here, the remainders satisfy
\begin{equation*}
    |r_N(x)|\;\leq\;\frac{1}{2}\omega_{f''}(1/N),
    \qquad x>0,
    \qquad
    |r_N(0)|\;\leq\;\frac{B}{2N}\|f''\|_\infty,
\end{equation*}
by the Taylor estimates in the proof of Lemma~\ref{lem:uniformGeneratorMBM}. Thus, $\msf{L}_Nf$ converges pointwise to $\mc{L}f$ for all $f \in \mf{D}(\mc{L})$ in the domain of the Sticky Brownian motion as $N \to \infty$.

\begin{lemma}
Let $f\in\mf{D}(\mc{L})$ and define $h_N(y):=\msf{L}_Nf(y)-\mc{L}f(y)$ for $y\in\tilde{\bb{G}}_N$, with $f(\Delta)=\mc{L}f(\Delta)=h_N(\Delta)=0$. Then
\begin{equation*}
    \sup_{y\in\tilde{\bb{G}}_N}|h_N(y)|
    \;\leq\;
    \max\bigg\{
        \frac{1}{2}\omega_{f''}(1/N),
        \frac{A}{N^{\alpha-2}}|f(0)|
        +\frac{B}{2N}\|f''\|_\infty
    \bigg\}
    \;\longrightarrow\;0.
\end{equation*}
In particular, there exists a constant $C_f<\infty$, independent of $N$, such that
\begin{equation*}
    \sup_{y\in\tilde{\bb{G}}_N}|\msf{L}_Nf(y)|
    \;\leq\;C_f.
\end{equation*}
\end{lemma}

\begin{proof}
The generator~\eqref{eq:generatorSBM} yields the following relation for the compensator~\eqref{eq:generatorMBM2}
\begin{equation*}
    h_N(y)
    \,=\,
    r_N(y)\one_{[y>0]}
    \,-\,
    \bigg(\frac{A}{N^{\alpha-2}}f(0)
    \,-\,r_N(0)\bigg)\one_{[y=0]}.
\end{equation*}
The estimates on $r_N$ give the claimed bound. Since $\alpha>2$, this bound converges to zero. The uniform bound on $\msf{L}_Nf$ follows because $\mc{L}f$ is bounded.
\end{proof}

\begin{lemma}\label{lem:characSBM}
Let $f \in \mf{D}(\mc{L})$. Along any subsequence for which $B^N\Rightarrow B$, $\mc{M}_{k_t}^{N,f} \Rightarrow \mc{M}_t^f$ as $N \to \infty$ in the Skorokhod space $\mc{D}([0,T], \bb{R})$.
Moreover,
\begin{equation*}
    \big(B^N,\mc{M}_{k_\cdot}^{N,f}\big)
    \;\Rightarrow\;
    \big(B,\mc{M}_\cdot^f\big)
\end{equation*}
in $\mc{D}([0,T],\bb{G})\times\mc{D}([0,T],\bb{R})$,
endowed with the product of the $J_1$ topologies.
\end{lemma}

\begin{proof}
The proof follows the same steps as in Lemma~\ref{lem:martingaleMBM}. By the preceding lemma,
\begin{equation*}
    \sup_{t\in[0,T]}
    \bigg|\sum_{j=0}^{k_t-1}
    h_N(B_{s_j}^{N})\frac{1}{N^2}\bigg|
    \;\leq\;
    T\sup_{y\in\tilde{\bb{G}}_N}|h_N(y)|
    \;\longrightarrow\;0.
\end{equation*}
Thus,
\begin{equation*}
    \sum_{j=0}^{k_t-1}
    \frac{h_N(B_{s_j}^{N})}{N^2}\to0
\end{equation*}
in probability.
The joint convergence follows as in Lemma~\ref{lem:martingaleMBM}.
\end{proof}

\normalcolor
\subsubsection{Exponential holding Brownian motion: the $(2,\beta,A,B)$ regime with $\beta>1$}

The exponential holding Brownian motion is a Feller process on the half-line, with continuous paths before its lifetime, characterized by the generator $\mc{L}$ and the domain $\mf{D}(\mc{L})$ as follows:
\begin{equation*}
    \mf{D}(\mc{L})\,=\,\bigg\{f\in\mc{C}_0^2(\bb{G}):f''\in\mc{C}_0([0,\infty)),\,\text{ and }\,Af(0)\,+\,\frac{1}{2}f''(0)\,=\,0\bigg\}.
\end{equation*}
In order to compare the discrete dynamics with the limiting continuous generator, fix an arbitrary test function $f \in \mf{D}(\mc{L})$. Once again, by~\eqref{eq:Brownian.generator}, we obtain, for $x \in \tilde{\bb{G}}_N\setminus\{\Delta\}$ that
\begin{equation}\label{eq:generatorEHBM}
\begin{split}
    \msf{L}_N\gn{f}{x}{}
    \,&=\,
    \bigg[\pfrac{1}{2}f''(x)
    \,+\,r_N(x)\bigg]\one_{[x>0]}
    \,+\,
    \bigg[\pfrac{B}{N^{\beta-1}}f'(0)
    \,+\,r_N(0)
    \,-\,A\f{0}\bigg]\one_{[x=0]}\\
    &=\,
    \big[\mc{L}\gn{f}{x}{}
    \,+\,r_N(x)\big]\one_{[x>0]}
    \,+\,
    \bigg[\pfrac{B}{N^{\beta-1}}f'(0)
    \,+\,\mc{L}\f{0}
    \,+\,r_N(0)\bigg]\one_{[x=0]}\,.
\end{split}
\end{equation}
Here, the Taylor estimates used in Lemma~\ref{lem:uniformGeneratorMBM} give
\begin{equation*}
    |r_N(x)|\;\leq\;\frac{1}{2}\omega_{f''}(1/N),
    \qquad x>0,
    \qquad
    |r_N(0)|\;\leq\;\frac{B}{2N^\beta}\|f''\|_\infty.
\end{equation*}
Thus, $\msf{L}_Nf$ converges pointwise to $\mc{L}f$ for all $f \in \mf{D}(\mc{L})$ in the domain of the exponential holding Brownian motion as $N \to \infty$.

\begin{lemma}
Let $f\in\mf{D}(\mc{L})$ and define $h_N(y):=\msf{L}_Nf(y)-\mc{L}f(y)$ for $y\in\tilde{\bb{G}}_N$, with $f(\Delta)=\mc{L}f(\Delta)=h_N(\Delta)=0$. Then
\begin{equation*}
    \sup_{y\in\tilde{\bb{G}}_N}|h_N(y)|
    \;\leq\;
    \max\bigg\{
        \frac{1}{2}\omega_{f''}(1/N),
        \frac{B}{N^{\beta-1}}|f'(0)|
        +\frac{B}{2N^\beta}\|f''\|_\infty
    \bigg\}
    \;\longrightarrow\;0.
\end{equation*}
\end{lemma}

\begin{proof}
The generator~\eqref{eq:generatorEHBM} yields
\begin{equation*}
    h_N(y)
    \,=\,
    r_N(y)\one_{[y>0]}
    \,+\,
    \bigg(\pfrac{B}{N^{\beta-1}}f'(0)
    \,+\,r_N(0)\bigg)\one_{[y=0]}.
\end{equation*}
The estimates on $r_N$ give the claimed bound.
Since $\beta>1$, both terms in the maximum converge to zero.
\end{proof}

\begin{lemma}\label{lem:characEHBM}
Let $f \in \mf{D}(\mc{L})$. Thus
$\mc{M}_{k_t}^{N,f} \Rightarrow \mc{M}_t^f$ as $N \to \infty$ in the Skorokhod space $\mc{D}([0,T], \bb{R})$. Moreover,
\begin{equation*}
    \big(B^N,\mc{M}_{k_\cdot}^{N,f}\big)
    \;\Rightarrow\;
    \big(B,\mc{M}_\cdot^f\big)
\end{equation*}
in $\mc{D}([0,T],\bb{G})\times\mc{D}([0,T],\bb{R})$,
endowed with the product of the $J_1$ topologies.
\end{lemma}

\begin{proof}
The proof follows the same steps as in Lemma~\ref{lem:martingaleMBM}. By the preceding lemma,
\begin{equation*}
    \sup_{t\in[0,T]}
    \bigg|\sum_{j=0}^{k_t-1}
    h_N(B_{s_j}^{N})\frac{1}{N^2}\bigg|
    \;\leq\;
    T\sup_{y\in\tilde{\bb{G}}_N}|h_N(y)|
    \;\longrightarrow\;0.
\end{equation*}
Thus, the error in the compensator converges to $0$ in probability,
\begin{equation*}
    \sum_{j=0}^{k_t-1}\frac{h_N(B_{s_j}^{N})}{N^2}\to0\,.
\end{equation*}
The convergence of the sequence of martingales
$\mc{M}_{k_t}^{N,f}$ to the process $\mc{M}_t^f$, jointly with $B^N$, follows from the same arguments as in Lemma~\ref{lem:martingaleMBM}.
\end{proof}

\normalcolor
\subsubsection{Absorbed Brownian motion: the $(\alpha,\beta,A,B)$ regime with $\alpha>2$ and $\beta>1$}

The absorbed Brownian motion is a Feller process whose generator $\mc{L} = \pfrac{1}{2}\pfrac{d^2}{dx^2}$ has its domain described via
\begin{equation*}
    \mf{D}(\mc{L})\,=\,\bigg\{f\in\mc{C}_0^2(\bb{G}):f''\in\mc{C}_0([0,\infty)),\,\text{ and }\,f''(0)\,=\,0\bigg\}.
\end{equation*}
In order to compare the discrete dynamics with the limiting continuous generator, fix an arbitrary test function $f \in \mf{D}(\mc{L})$. From equation~\eqref{eq:Brownian.generator} and the boundary condition on $f$, we obtain, for $x\in\tilde{\bb{G}}_N\setminus\{\Delta\}$ that
\begin{equation}\label{eq:generatorABM}
\begin{split}
    \msf{L}_N\gn{f}{x}{}
    \,&=\,
    \big(\mc{L}\f{x}\,+\,r_N(x)\big)\one_{[x>0]}\\
    &\quad+\,
    \bigg(\pfrac{B}{N^{\beta-1}}f'(0)
    \,-\,\pfrac{A}{N^{\alpha-2}}f(0)
    \,+\,r_N(0)\bigg)\one_{[x=0]}.
\end{split}
\end{equation}
Here, the Taylor estimates used in Lemma~\ref{lem:uniformGeneratorMBM} give
\begin{equation*}
    |r_N(x)|\;\leq\;\frac{1}{2}\omega_{f''}(1/N),
    \qquad x>0,
    \qquad
    |r_N(0)|\;\leq\;\frac{B}{2N^\beta}\|f''\|_\infty.
\end{equation*}
Thus, for $\alpha > 2$ and $\beta > 1$,
$\msf{L}_Nf$ converges pointwise to $\mc{L}f$ for all $f \in \mf{D}(\mc{L})$ in the domain of the absorbed Brownian motion as $N \to \infty$.

\begin{lemma}
Let $f\in\mf{D}(\mc{L})$ and define $h_N(x):=\msf{L}_Nf(x)-\mc{L}f(x)$ for $x\in\tilde{\bb{G}}_N$, with $f(\Delta)=\mc{L}f(\Delta)=h_N(\Delta)=0$. Then
\begin{equation*}
    \sup_{x\in\tilde{\bb{G}}_N}|h_N(x)|
    \;\leq\;
    \max\bigg\{
        \frac{1}{2}\omega_{f''}(1/N),
        \frac{B}{N^{\beta-1}}|f'(0)|
        +\frac{A}{N^{\alpha-2}}|f(0)|
        +\frac{B}{2N^\beta}\|f''\|_\infty
    \bigg\}
    \;\longrightarrow\;0.
\end{equation*}
\end{lemma}

\begin{proof}
Since $\mc{L}f(0)=0$, the discrete generator~\eqref{eq:generatorABM} yields
\begin{equation}\label{eq:generatorABM2}
\begin{split}
    h_N(x)
    \,&=\,
    r_N(x)\one_{[x>0]}
    \,+\,
    \bigg[\frac{B}{N^{\beta-1}}f'(0)
    \,-\,\frac{A}{N^{\alpha-2}}f(0)
    \,+\,r_N(0)\bigg]\one_{[x=0]}.
\end{split}
\end{equation}
The estimates on $r_N$ give the claimed bound.
Since $\alpha>2$ and $\beta>1$, all terms in this bound converge to zero.
\end{proof}

\begin{lemma}\label{lem:characABM}
Let $f \in \mf{D}(\mc{L})$. Thus
$\mc{M}_{k_t}^{N,f} \Rightarrow \mc{M}_t^f$ as $N \to \infty$ in the Skorokhod space $\mc{D}([0,T], \bb{R})$. Moreover,
\begin{equation*}
    \big(B^N,\mc{M}_{k_\cdot}^{N,f}\big)
    \;\Rightarrow\;
    \big(B,\mc{M}_\cdot^f\big)
\end{equation*}
in $\mc{D}([0,T],\bb{G})\times\mc{D}([0,T],\bb{R})$, endowed with the product of the $J_1$ topologies.
\end{lemma}

\begin{proof}
The proof follows the same steps as in Lemma~\ref{lem:martingaleMBM}. By the preceding lemma,
\begin{equation*}
    \sup_{t\in[0,T]}
    \bigg|\sum_{j=0}^{k_t-1}
    h_N(B_{s_j}^{N})\frac{1}{N^2}\bigg|
    \;\leq\;
    T\sup_{x\in\tilde{\bb{G}}_N}|h_N(x)|
    \;\longrightarrow\;0.
\end{equation*}
Thus, we have that
\begin{equation*}
    \sum_{j=0}^{k_t-1}
    \frac{h_N(B_{s_j}^{N})}{N^2}\,\to\,0\,,
\end{equation*}
in probability. Thus the convergence of $\mc{M}_{k_t}^{N,f}$ to the process $\mc{M}_t^f$ holds. The joint convergence with $B^N$ follows from the same arguments as in Lemma~\ref{lem:martingaleMBM}.

\end{proof}

The proof for the next cases follows the same approach as in Lemma~\ref{lem:martingaleMBM}, with the key difference being the need to control the local time in order to obtain the required upper bound on the compensator.

This control is achieved by invoking the upper bound established in Section~\ref{sec:5}, via Corollary~\ref{cor:localTimeBound}, specifically~\eqref{eq:localTimeDiffusiveBound}. Indeed, since $B_{s_j}^N=X_j^N/N$ for $s_j=j/N^2$, the number of visits to the origin satisfies
\begin{equation*}
    \sum_{j=0}^{k_t-1}\one_{[B_{s_j}^N=0]}
    \,=\,
    \sum_{j=0}^{k_t-1}\one_{[X_j^N=0]}
    \,\leq\,
    \mf{L}_{k_t}^{X^N}(0).
\end{equation*}
Consequently,
\begin{equation*}
    \sup_{t\in[0,T]}
    \bb{E}^N\bigg[
        \sum_{j=0}^{k_t-1}\one_{[B_{s_j}^N=0]}
    \bigg]
    \;\leq\;
    2+\frac{2D_*\sqrt{T}}{B\sqrt{2\pi}}N^{\beta+1}.
\end{equation*}
This is the estimate used below to control the boundary contribution to the compensator.

\normalcolor
\subsubsection{Reflected Brownian motion: the $(\alpha,\beta,A,B)$ regime with $\alpha>1+\beta$ and $\beta\in[0,1)$}

The reflected Brownian motion is a Feller diffusion process on the half-line, characterized by the generator $\mc{L}$ and the domain $\mf{D}(\mc{L})$ given by
\begin{equation*}
    \mf{D}(\mc{L})\,=\,\bigg\{f\in\mc{C}_0^2(\bb{G}):f''\in\mc{C}_0([0,\infty)),\,\text{ and }\,f'(0)\,=\,0\bigg\}.
\end{equation*}
In order to compare the discrete dynamics with the limiting continuous generator, fix an arbitrary test function $f \in \mf{D}(\mc{L})$. From the generator of boundary random walk~\eqref{eq:Brownian.generator} and choosing $\beta \in [0,1)$ and $\alpha > 1 + \beta$ in addition to the boundary conditions over $f$, we obtain, for $x \in \tilde{\bb{G}}_N\setminus\{\Delta\}$ that
\begin{equation}\label{eq:generatorRBM}
\begin{split}
    \msf{L}_N\gn{f}{x}{}
    \,&=\,
    \big[\mc{L}\gn{f}{x}{}
    \,+\,r_N(x)\big]\one_{[x>0]}
    \,+\,
    \bigg[-\,\pfrac{A}{N^{\alpha-2}}f(0)
    \,+\,r_N(0)\bigg]\one_{[x=0]}.
\end{split}
\end{equation}
Here, the Taylor estimates used in Lemma~\ref{lem:uniformGeneratorMBM} give
\begin{equation*}
    |r_N(x)|\;\leq\;\frac{1}{2}\omega_{f''}(1/N),
    \qquad x>0,
    \qquad
    |r_N(0)|\;\leq\;\frac{B}{2N^\beta}\|f''\|_\infty.
\end{equation*}
Unlike the preceding regimes, the generator error at the origin need not converge to zero. We therefore use the local-time bound to control its contribution to the compensator.

\begin{lemma}
Let $f\in\mf{D}(\mc{L})$ and define $h_N(y) := \msf{L}_Nf(y) - \mc{L}f(y)$ for $y\in\tilde{\bb{G}}_N$, with $f(\Delta) = \mc{L}f(\Delta) = h_N(\Delta) = 0$. Then
\begin{equation*}
    \sup_{t\in[0,T]}
    \bigg|\sum_{j=0}^{k_t-1}
    \frac{h_N(B_{s_j}^{N})}{N^2}\bigg|
    \;\longrightarrow\;0
\end{equation*}
in probability.
\end{lemma}

\begin{proof}
Let $f\in\mf{D}(\mc{L})$ fixed on the domain of the reflected Brownian motion. From~\eqref{eq:generatorRBM} we have that
\begin{equation}\label{eq:generatorRBM2}
    h_N(y)
    \,=\,
    r_N(y)\one_{[y>0]}
    \,+\,
    \bigg[-\,\pfrac{A}{N^{\alpha-2}}f(0)
    \,-\,\mc{L}\f{0}
    \,+\,r_N(0)\bigg]\one_{[y=0]}.
\end{equation}
Consequently,
\begin{equation*}
\begin{split}
    \sup_{t\in[0,T]}
    \bigg|\sum_{j=0}^{k_t-1}
    \frac{h_N(B_{s_j}^{N})}{N^2}\bigg|
    &\leq
    \frac{T}{2}\omega_{f''}(1/N)\\
    &\quad+
    \frac{1}{N^2}
    \bigg[
        \frac{A}{N^{\alpha-2}}|f(0)|
        +|\mc{L}f(0)|
        +\frac{B}{2N^\beta}\|f''\|_\infty
    \bigg]\mf{L}_{k_T}^{X^N}(0).
\end{split}
\end{equation*}
By Corollary~\ref{cor:localTimeBound}, there exists a constant $C_T<\infty$, independent of $N$, such that
\begin{equation*}
    \bb{E}^N\big[\mf{L}_{k_T}^{X^N}(0)\big]
    \;\leq\;C_TN^{1+\beta}.
\end{equation*}
Thus, Markov's inequality gives
\begin{equation*}
\begin{split}
    &\Prob{}{
        \sup_{t\in[0,T]}
        \bigg|\sum_{j=0}^{k_t-1}
        \frac{h_N(B_{s_j}^{N})}{N^2}\bigg|>\eps
    }\\
    &\qquad\leq
    \frac{1}{\eps}\bigg[
        \frac{T}{2}\omega_{f''}(1/N)
        +C_TA|f(0)|N^{1+\beta-\alpha}
        +C_T|\mc{L}f(0)|N^{\beta-1}
        +\frac{C_TB}{2N}\|f''\|_\infty
    \bigg].
\end{split}
\end{equation*}
Since $\alpha > 1 + \beta$ and $0 \leq \beta < 1$, the right-hand side converges to $0$ as $N \to \infty$, proving the result.
\end{proof}

\begin{lemma}\label{lem:characRBM}
Let $f \in \mf{D}(\mc{L})$. Thus $\mc{M}_{k_t}^{N,f} \Rightarrow \mc{M}_t^f$ as $N \to \infty$ in the Skorokhod space $\mc{D}([0,T], \bb{R})$.
Moreover,
\begin{equation*}
    \big(B^N,\mc{M}_{k_\cdot}^{N,f}\big)
    \;\Rightarrow\;
    \big(B,\mc{M}_\cdot^f\big)
\end{equation*}
in $\mc{D}([0,T],\bb{G})\times\mc{D}([0,T],\bb{R})$, endowed with the product of the $J_1$ topologies.
\end{lemma}

\begin{proof}
By the preceding lemma, we have that
\begin{equation*}
    \sum_{j=0}^{k_t-1}
    \frac{h_N(B_{s_j}^{N})}{N^2}\to0
\end{equation*}
in probability, uniformly over $t\in[0,T]$. Once again, $\mc{M}_{k_t}^{N,f}$ converges to the process $\mc{M}_t^f$
by the same arguments as in Lemma~\ref{lem:martingaleMBM}. The same argument also gives the joint convergence with $B^N$.

\end{proof}

\normalcolor
\subsubsection{Elastic Brownian motion: the $(1+\beta,\beta,A,B)$ regime with $\beta\in[0,1)$}

The elastic Brownian motion is a Feller process on the half-line, with continuous paths before its lifetime, characterized by the generator $\mc{L} = \pfrac{1}{2}\pfrac{d^2}{dx^2}$ whose domain is given as follows:
\begin{equation*}
    \mf{D}(\mc{L})\,=\,\bigg\{f\in\mc{C}_0^2(\bb{G}):f''\in\mc{C}_0([0,\infty)),\,\text{ and }\,Af(0)\,-\,Bf'(0)\,=\,0\bigg\}.
\end{equation*}
In order to compare the discrete dynamics with the limiting continuous generator, fix an arbitrary test function $f \in \mf{D}(\mc{L})$. From the generator of boundary random walk~\eqref{eq:Brownian.generator}, for $\alpha = 1 + \beta$ we obtain, for $x \in \tilde{\bb{G}}_N\setminus\{\Delta\}$ the following
\begin{equation}\label{eq:generatorEBM}
\begin{split}
    \msf{L}_N\gn{f}{x}{}
    \,&=\,
    \bigg(\pfrac{1}{2}f''(x)
    \,+\,r_N(x)\bigg)\one_{[x>0]} \;+\,
    \bigg(\pfrac{B}{N^{\beta-1}}f'(0)
    \,-\,\pfrac{A}{N^{\beta-1}}f(0)
    \,+\,r_N(0)\bigg)\one_{[x=0]}\\
    &=\,
    \big(\mc{L}\gn{f}{x}{}
    \,+\,r_N(x)\big)\one_{[x>0]}
    \,+\,
    \bigg(\pfrac{1}{N^{\beta-1}}
    \big[Bf'(0)-Af(0)\big]
    \,+\,r_N(0)\bigg)\one_{[x=0]}\\
    &=\,
    \big(\mc{L}\gn{f}{x}{}
    \,+\,r_N(x)\big)\one_{[x>0]}
    \,+\,r_N(0)\one_{[x=0]}.
\end{split}
\end{equation}
Here, the Taylor estimates used in Lemma~\ref{lem:uniformGeneratorMBM} give
\begin{equation*}
    |r_N(x)|\;\leq\;\frac{1}{2}\omega_{f''}(1/N),
    \qquad x>0,
    \qquad
    |r_N(0)|\;\leq\;\frac{B}{2N^\beta}\|f''\|_\infty.
\end{equation*}
As in the reflected case, we control the boundary contribution to the compensator using the local-time estimate.

\begin{lemma}
Let $f\in\mf{D}(\mc{L})$ and define $h_N(x):=\msf{L}_Nf(x)-\mc{L}f(x)$ for $x\in\tilde{\bb{G}}_N$, with $f(\Delta)=\mc{L}f(\Delta)=h_N(\Delta)=0$. Then
\begin{equation*}
    \sup_{t\in[0,T]}
    \bigg|\sum_{j=0}^{k_t-1}
    \frac{h_N(B_{s_j}^{N})}{N^2}\bigg|
    \;\longrightarrow\;0
\end{equation*}
in probability.
\end{lemma}

\begin{proof}
From the boundary conditions of the limit process, the discrete generator can be rewritten as in~\eqref{eq:generatorEBM}, and therefore it yields
\begin{equation*}
    h_N(x)
    \,=\,
    r_N(x)\one_{[x>0]}
    \,-\,
    \big(\mc{L}f(0)
    \,-\, r_N(0)\big)\one_{[x=0]}.
\end{equation*}
Consequently,
\begin{equation*}
\begin{split}
    \sup_{t\in[0,T]}
    \bigg|\sum_{j=0}^{k_t-1}
    \frac{h_N(B_{s_j}^{N})}{N^2}\bigg|
    &\leq
    \frac{T}{2}\omega_{f''}(1/N)\\
    &\quad+
    \frac{1}{N^2}
    \bigg[
        |\mc{L}f(0)|
        +\frac{B}{2N^\beta}\|f''\|_\infty
    \bigg]\mf{L}_{k_T}^{X^N}(0).
\end{split}
\end{equation*}
By Corollary~\ref{cor:localTimeBound}, there exists a constant $C_T<\infty$, independent of $N$, such that
\begin{equation*}
    \bb{E}^N\big[\mf{L}_{k_T}^{X^N}(0)\big]
    \;\leq\;C_TN^{1+\beta}.
\end{equation*}
Thus, Markov's inequality gives
\begin{equation*}
\begin{split}
    &\Prob{}{
        \sup_{t\in[0,T]}
        \bigg|\sum_{j=0}^{k_t-1}
        \frac{h_N(B_{s_j}^{N})}{N^2}\bigg|>\eps
    }\\
    &\qquad\leq
    \frac{1}{\eps}\bigg[
        \frac{T}{2}\omega_{f''}(1/N)
        +C_T|\mc{L}f(0)|N^{\beta-1}
        +\frac{C_TB}{2N}\|f''\|_\infty
    \bigg].
\end{split}
\end{equation*}
Since $\beta\in[0,1)$, the right-hand side converges to zero, which proves the result.
\end{proof}

\begin{lemma}\label{lem:characEBM}
Let $f \in \mf{D}(\mc{L})$. Thus $\mc{M}_{k_t}^{N,f} \Rightarrow \mc{M}_t^f$ as $N \to \infty$ in the Skorokhod space $\mc{D}([0,T], \bb{R})$. Moreover,
\begin{equation*}
    \big(B^N,\mc{M}_{k_\cdot}^{N,f}\big)
    \;\Rightarrow\;
    \big(B,\mc{M}_\cdot^f\big)
\end{equation*}
in $\mc{D}([0,T],\bb{G})\times\mc{D}([0,T],\bb{R})$,
endowed with the product of the $J_1$ topologies.
\end{lemma}

\begin{proof}
By the preceding lemma, we have that
\begin{equation*}
    \sum_{j=0}^{k_t-1}
    \frac{h_N(B_{s_j}^{N})}{N^2}\to0
\end{equation*}
in probability, uniformly over $t\in[0,T]$. Thus, the convergence of the sequence of martingales $\mc{M}_{k_t}^{N,f}$ to the process $\mc{M}_t^f$ holds, by the same arguments as in Lemma~\ref{lem:martingaleMBM}. The same argument also gives the joint convergence with $B^N$.
\end{proof}

\subsection{Uniform integrability of the discrete martingales}

\normalcolor
The preceding subsections establish the joint convergence of the rescaled random walks and their associated discrete martingales. To pass the martingale property to the limit, we next prove uniform integrability of the discrete martingales. This will allow us to pass to the limit in the corresponding expectation identities and identify the limiting processes as solutions to the respective martingale problems.

\begin{lemma}\label{lem:uniformIntegrabilityMBM}
Assume $\beta\geq0$ and $\alpha\geq\min\{2,\beta+1\}$. For every $f\in\mf{D}(\mc{L})$, the family
$\{\mc{M}_{k_t}^{N,f}:N,\ t\in[0,T]\}$ is uniformly integrable.
\end{lemma}

\begin{proof}
Fix $f\in\mf{D}(\mc{L})$. We set $f(\Delta)=0$ and write $f_N(x)=f(x/N)$ for $x\in\tilde{\bb{G}}\setminus\{\Delta\}$, with $f_N(\Delta) = 0$.

Denote by $\Lambda_N$ the carré du champ operator
associated with $\hat{\msf{L}}_N$, written in rescaled coordinates
\begin{equation*}
\begin{split}
    \Lambda_Nf(x/N)
    \;&:=\;
    \sum_{y\in{\tilde{\bb{G}}}}
    p_N(x,y)\big(f_N(y)-f_N(x)\big)^2\\
    &\leq\;
    \pfrac{1}{N^2}\|f'\|_\infty^2\one_{[x>0]}
    \;+\;
    \bigg(
        \pfrac{B{\|f'\|_\infty^2}}{N^{\beta+2}}
        \;+\;
        \pfrac{A{|f(0)|^2}}{N^\alpha}
    \bigg)\one_{[x=0]}\,,
\end{split}
\end{equation*}
for $x\neq\Delta$, and $\Lambda_Nf(\Delta)=0$. Here, $p_N$ denotes the transition probabilities of the discrete process $X^N$. Since $\sum p_N(x,y)\leq1$, the Cauchy-Schwarz inequality
gives
\begin{equation}\label{eq:uniform_integrability3}
\begin{split}
    \big(\hat{\msf{L}}_Nf_N(x)\big)^2
    \;&=\;
    \bigg(
        \sum_{y\in\tilde{\bb{G}}}
        p_N(x,y)\big[f_N(y)-f_N(x)\big]
    \bigg)^2\\
    &\leq\;
    \sum_{y\in\tilde{\bb{G}}}p_N(x,y)
    \cdot
    \sum_{y\in\tilde{\bb{G}}}
    p_N(x,y)\big[f_N(y)-f_N(x)\big]^2\\
    &\leq\;\Lambda_Nf\big(\pfrac{1}{N}x\big).
\end{split}
\end{equation}
The same inequality holds at $\Delta$, where both sides vanish.

From the orthogonality of the martingale increments, together with
the tower property and the conditional variance identity,
we have

\begin{equation*}
\begin{split}
    \bb{E}\Big[\big(\mc{M}_{k_t}^{N,f}\big)^2\Big]
    \;&=\;
    \bb{E}\Big[
        \sum_{j=0}^{k_t-1}
        \bb{E}\big[
            \big(\mc{M}_{j+1}^{N,f}-\mc{M}_j^{N,f}\big)^2
            \;\big|\;\mc{F}_j^{X^N}
        \big]
    \Big]\\
    &=\;
    \sum_{j=0}^{k_t-1}
    \bb{E}\Big[
        \Lambda_Nf(B_{s_j}^N)
        -\big(\hat{\msf{L}}_Nf_N(X_j^N)\big)^2
    \Big]\\
    &\leq\;
    \sum_{j=0}^{k_t-1}
    \bb{E}\big[\Lambda_Nf(B_{s_j}^N)\big]\,.
\end{split}
\end{equation*}
Therefore, we can bound the second moment of the martingale by
\begin{equation}\label{eq:square_martingale_bound}
\begin{split}
    \bb{E}\Big[\big(\mc{M}_{k_t}^{N,f}\big)^2\Big]
    &\leq\;
    \sum_{j=0}^{k_T-1}
    \bb{E}\Big[\Lambda_Nf(B_{s_j}^N)\Big]\\
    &\leq\;
    \sum_{j=0}^{k_T-1}
    \bb{E}\Big[
        \pfrac{1}{N^2}\|f'\|_\infty^2\one_{[B_{s_j}^N>0]}
        +
        \bigg(
            \pfrac{B\|f'\|_\infty^2}{N^{\beta+2}}
            +
            \pfrac{A|f(0)|^2}{N^\alpha}
        \bigg)\one_{[B_{s_j}^N=0]}
    \Big]\\
    &\leq\;
    \pfrac{1}{N^2}\|f'\|_\infty^2 k_T
    +
    \bigg(
        \pfrac{B\|f'\|_\infty^2}{N^{\beta+2}}
        +
        \pfrac{A|f(0)|^2}{N^\alpha}
    \bigg)
    \bb{E}\Big[\mathfrak{L}_{k_T-1}^{X^N}(0)\Big].
\end{split}
\end{equation}
When $k_T=0$, the sums are empty and the martingale is zero,
so we only need to consider $k_T\geq1$.

Observe that, from Corollary~\ref{cor:localTimeBound}, we have that $\bb{E}[\mathfrak{L}_{k_T-1}^{X^N}(0)] \leq C_TN^{\beta+1}$ for some constant $C_T>0$ independent of $N$.

Since $k_T\leq N^2T$, equation~\eqref{eq:square_martingale_bound} yields
\begin{equation*}
    \bb{E}\Big[\big|\mc{M}_{k_t}^{N,f}\big|^2\Big]
    \;\leq\;
    T\|f'\|_\infty^2
    +\frac{C_TB}{N}\|f'\|_\infty^2
    +C_TA|f(0)|^2N^{\beta+1-\alpha}.
\end{equation*}
This bound is uniform in $N$ and $t\in[0,T]$ whenever $\beta+1\leq\alpha$.

Alternatively, bounding the number of visits to the origin by the total number of steps $k_T$, we obtain
\begin{equation*}
    \bb{E}\Big[\big|\mc{M}_{k_t}^{N,f}\big|^2\Big]
    \;\leq\;
    \pfrac{1}{N^2}\|f'\|_\infty^2 k_T
    +
    \bigg(
        \pfrac{B\|f'\|_\infty^2}{N^{\beta+2}}
        +
        \pfrac{A|f(0)|^2}{N^\alpha}
    \bigg)k_T.
\end{equation*}
Since $k_T\leq N^2T$ and $\beta\geq0$, this bound is uniform in $N$ whenever $\alpha\geq2$. Combining the two estimates, we conclude that
\begin{equation*}
    \sup_N\sup_{t\in[0,T]}
    \bb{E}\Big[\big|\mc{M}_{k_t}^{N,f}\big|^2\Big]
    \;<\;\infty
\end{equation*}
whenever $\alpha\geq\min\{2,\beta+1\}$. Hence, the family $\{\mc{M}_{k_t}^{N,f}:N,\ t\in[0,T]\}$ is uniformly integrable.
\end{proof}

The next result establishes that the limit process $\mc{M}_t^f$ is a martingale with respect to the filtration generated by the limit process $B$.

\begin{lemma}\label{lem:UnifMartingale}
Let $f \in \mf{D}(\mc{L})$ for any of the mixed, sticky, exponential holding, absorbed, reflected, or elastic regimes. Then $\mc{M}_t^f$ is a martingale with respect to the filtration generated by $B$.
\end{lemma}

\begin{proof}

By Lemma~\ref{lem:discreteMartingale}, $\mc{M}_{k_t}^{N,f}$ is a martingale.

We use the joint convergence established in
Lemmas~\ref{lem:martingaleMBM}, \ref{lem:characSBM},
\ref{lem:characEHBM}, \ref{lem:characABM},
\ref{lem:characRBM}, and~\ref{lem:characEBM},
according to the regime under consideration, together with
Lemma~\ref{lem:uniformIntegrabilityMBM}, to pass this property
to the limit. The parameter condition in the latter lemma
holds in each of these regimes.

Write $\mc{F}_t^B:=\sigma(B_r:0\leq r\leq t)$ and define
\begin{equation*}
    \mc{T} \;:=\; \{0,T\} \;\cup\;
    \big\{u\in(0,T):\bb{P}(B_u=B_{u-})=1\big\}\,.
\end{equation*}
The complement of $\mc{T}$ is at most countable, and hence $\mc{T}$ is dense in $[0,T]$. At its interior times, evaluation is almost surely continuous under the law of $B$ for the $J_1$ topology. Evaluation at either endpoint is continuous for every path. Moreover, $\mc{M}^f$ is almost surely continuous at each interior time in $\mc{T}$, since $f$ is continuous and the integral in~\eqref{eq:martingale} is continuous in time.

Let $t_1<\cdots<t_m\leq s<t$ belong to $\mc{T}$ and fix $H\in\mc{C}_b(\bb{G}^m)$. Since $H(B_{t_1}^N,\ldots,B_{t_m}^N)$ is $\mc{F}_{k_s}^{X^N}$-measurable, the discrete martingale property gives
\begin{equation}\label{eq:martingaleIdentityDiscreteMBM}
    \bb{E}^N\big[
        \big(\mc{M}_{k_t}^{N,f}-\mc{M}_{k_s}^{N,f}\big)
        H(B_{t_1}^N,\ldots,B_{t_m}^N)
    \big] \;=\; 0\,.
\end{equation}
By the corresponding joint convergence and the mapping theorem,
\begin{equation*}
    \big(\mc{M}_{k_t}^{N,f}-\mc{M}_{k_s}^{N,f}\big)
    H(B_{t_1}^N,\ldots,B_{t_m}^N)
    \;\Rightarrow\;
    \big(\mc{M}_t^f-\mc{M}_s^f\big)
    H(B_{t_1},\ldots,B_{t_m})\,.
\end{equation*}
The bound in Lemma~\ref{lem:uniformIntegrabilityMBM} gives

\begin{equation*}
\begin{split}
    &\sup_N\bb{E}^N\Big[
        \big|
            \big(\mc{M}_{k_t}^{N,f}-\mc{M}_{k_s}^{N,f}\big)
            H(B_{t_1}^N,\ldots,B_{t_m}^N)
        \big|^2
    \Big]\\
    &\qquad\leq\;
    4\|H\|_\infty^2
    \sup_N\sup_{u\in[0,T]}
    \bb{E}^N\big[|\mc{M}_{k_u}^{N,f}|^2\big]
    \;<\;\infty.
\end{split}
\end{equation*}
Thus, these random variables are uniformly integrable. Passing to the limit in~\eqref{eq:martingaleIdentityDiscreteMBM}, we obtain
\begin{equation}\label{eq:martingaleIdentityLimitMBM}
    \bb{E}\big[
        \big(\mc{M}_t^f-\mc{M}_s^f\big)
        H(B_{t_1},\ldots,B_{t_m})
    \big] \;=\; 0\,.
\end{equation}

For $s\in\mc{T}$, the evaluations at times in $\mc{T}\cap[0,s]$ generate $\mc{F}_s^B$. Indeed, for every $r<s$, choose $r_n\in\mc{T}\cap[0,s]$ with $r_n\downarrow r$. Right-continuity gives $B_{r_n}\to B_r$, and the evaluation at $s$ is already included. Bounded continuous cylinder functions of these evaluations form an algebra generating $\mc{F}_s^B$. By the monotone class theorem, identity~\eqref{eq:martingaleIdentityLimitMBM} therefore holds with $H(B_{t_1},\ldots,B_{t_m})$ replaced by any bounded $\mc{F}_s^B$-measurable random variable $Z$, provided $s,t\in\mc{T}$.

It remains to extend the identity to arbitrary $0\leq s<t\leq T$. Choose $s_n\in\mc{T}$ with $s_n\downarrow s$ and $s_n<t$. If $t<T$, choose $t_n\in\mc{T}$ with $t_n\downarrow t$; if $t=T$, set $t_n=T$. Since $\mc{F}_s^B\subseteq\mc{F}_{s_n}^B$, every bounded $\mc{F}_s^B$-measurable random variable $Z$ satisfies
\begin{equation*}
    \bb{E}\big[
        \big(\mc{M}_{t_n}^f-\mc{M}_{s_n}^f\big)Z
    \big] \;=\; 0\,.
\end{equation*}
The process $\mc{M}^f$ has càdlàg paths and, by its definition,
\begin{equation*}
    \sup_{u\in[0,T]}|\mc{M}_u^f|
    \;\leq\;
    2\|f\|_\infty+T\|\mc{L}f\|_\infty\,.
\end{equation*}
Right-continuity and dominated convergence now yield
\begin{equation*}
    \bb{E}\big[
        \big(\mc{M}_t^f-\mc{M}_s^f\big)Z
    \big] \;=\; 0\,.
\end{equation*}
Therefore, $\mc{M}^f$ is a martingale with respect to the natural filtration of $B$.
\end{proof}

\subsection{Well-posedness of the Martingale Problem for the general Brownian motion}
\label{subsec:wellposedness}

Once tightness has been established (Lemma~\ref{lem:tight}) and the convergence of the associated martingales has been demonstrated throughout this section, we conclude that the scaling limit solves the martingale problem associated with the generator~$\mc{L}$, as described in Theorem~\ref{thm:Fellerprocess}, corresponding to a choice of coefficients $c_i \geq 0$ for $i \in \{1,2,3\}$, with $c_1+c_2+c_3=1$ and $c_1\neq1$. The resulting process is a Feller process~\cite[Theorem 1.3, page 3]{KosVadPot10}, and hence the martingale problem is well-posed (see~\cite[Corollary 6.2.5, page 151]{StroockVaradham} or~\cite[Proposition 2.4, page 283]{vanCasteren92}): it admits a unique solution, which defines the general Brownian motion on the half-line.

\section{Proof of the main theorem}\label{sec:6}

We now combine the tightness estimates from Section~\ref{sec:3}, the identification of limit points from Section~\ref{sec:4}, and the well-posedness of the limiting martingale problems established in Subsection~\ref{subsec:wellposedness} to prove Theorem~\ref{thm:main}.

\begin{proof}[Proof of Theorem~\ref{thm:main}]

Fix $T>0$ and one of the six parameter regimes in the statement. Let $(\mc{L},\mf{D}(\mc{L}))$ be the corresponding limiting generator described in Subsection~\ref{subsec:compensatorConvergence} and by $\xi$ we denote the limiting process. In each of these regimes, $\alpha\geq\min\{2,\beta+1\}$.

We first justify tightness under the assumed convergence in
distribution of the initial states. For $R>0$, let $B^{N,R}$ have the same transition probabilities as $B^N$, with initial state
\begin{equation*}
    B_0^{N,R} \;=\;
    \begin{cases}
        B_0^N, & B_0^N\in[0,R]\cup\{\Delta\},\\
        0, & B_0^N\in(R,\infty).
    \end{cases}
\end{equation*}
The spatial initial states of $B^{N,R}$ are bounded by $R$.
Hence, Lemma~\ref{lem:tight} implies that, for each fixed $R$, the laws of $B^{N,R}$ are tight in $\mc{D}([0,T],\bb{G})$.

Couple $B^{N,R}$ and $B^N$ so that their paths agree whenever their initial states agree. For every measurable $E\subseteq\mc{D}([0,T],\bb{G})$, this gives
\begin{equation*}
    \big|\bb{P}(B^N\in E)-\bb{P}(B^{N,R}\in E)\big| \;\leq\; \bb{P}(B_0^N\in(R,\infty)).
\end{equation*}
Since $B_0^N\Rightarrow\xi$, the initial laws are tight, and therefore
\begin{equation*}
    \lim_{R\to\infty}\sup_N\bb{P}(B_0^N\in(R,\infty)) \;=\; 0.
\end{equation*}
Combining this estimate with tightness of the truncated processes proves tightness of the laws of $B^N$.

Let $\{\bb{P}^{N_j}:j\in\bb{N}\}$ be an arbitrary subsequence of the laws of $B^N$. By tightness, there exist a further subsequence, which we do not relabel, and a probability measure $\bb{P}$ on $\mc{D}([0,T],\bb{G})$ such that
\begin{equation*}
    \bb{P}^{N_j} \;\Rightarrow\; \bb{P}.
\end{equation*}
Let $B$ denote the canonical process under $\bb{P}$.
Since evaluation at time zero is continuous for the $J_1$ topology and $B_0^N\Rightarrow\xi$, the distribution of $B_0$ under $\bb{P}$ is the distribution of $\xi$.

By Lemma~\ref{lem:UnifMartingale}, for every $f\in\mf{D}(\mc{L})$, the process
\begin{equation*}
    \mc{M}_t^f \;=\; f(B_t)-f(B_0)-\int_0^t\mc{L}f(B_s)\dd s, \qquad t\in[0,T],
\end{equation*}
is a martingale with respect to the natural filtration of $B$. Thus, $\bb{P}$ solves the martingale problem associated with $(\mc{L},\mf{D}(\mc{L}))$, with initial distribution equal to the law of $\xi$.

By the well-posedness established in Subsection~\ref{subsec:wellposedness}, this martingale problem has a unique solution. Consequently, $\bb{P}$ is the law of the
corresponding mixed, sticky, exponential holding, absorbed,
reflected, or elastic Brownian motion, with the initial
distribution and parameters specified in
Theorem~\ref{thm:main}.

We have therefore shown that every subsequence of $\{\bb{P}^N:N\in\bb{N}\}$ admits a further subsequence converging to the same probability measure $\bb{P}$. It follows that
\begin{equation*}
    \bb{P}^N\;\Rightarrow\;\bb{P}
\end{equation*}
in $\mc{D}([0,T],\bb{G})$, endowed with the $J_1$ topology.
This proves all six assertions of Theorem~\ref{thm:main}.

\end{proof}

\section*{Acknowledgements}
The authors would also like to thank Ying Li (Xiangtan University, China) for kindly pointing out corrections and typographical errors in a previous version of this work.

The authors also thank the anonymous referee for the careful reading of the manuscript and for the detailed comments, which helped improve the proofs and the presentation of the results.

\bibliographystyle{plain}
\bibliography{bibliografia}

\end{document}